\documentclass[11pt]{article}

\usepackage[applemac]{inputenc}	
\usepackage[T1]{fontenc}
\usepackage[english]{babel}	
\usepackage[a4paper,DIV=11]{typearea}
\usepackage{parskip}
\usepackage{amsmath,amsfonts,amssymb,amsthm, verbatim}		
\usepackage{mathtools}                          
\usepackage{color}      
\usepackage{graphicx}
\usepackage{booktabs}
\usepackage{stmaryrd}  

\newcommand{\ds}{\displaystyle}

\newcommand{\R}{\mathbb{R}}	
\newcommand{\N}{\mathbb{N}}	
\newcommand{\Z}{\mathbb{Z}}	

\newcommand{\FF}{\mathcal{S}_0}

\newcommand{\LL}{\mathcal{L}}
\newcommand{\AAA}{\mathcal{A}}
\newcommand{\nd}{\noindent}

\newcommand{\ttt}{{\tilde T}}

\newcommand{\eps}{\varepsilon}

\newcommand{\beq}{\begin{equation}}
\newcommand{\eeq}{\end{equation}}

\newcommand{\wto}{\xrightharpoonup{}}          

\def\XXint#1#2#3{{\setbox0=\hbox{$#1{#2#3}{\int}$}
\vcenter{\hbox{$#2#3$}}\kern-.5\wd0}}

\DeclareMathOperator{\sgn}{sgn}

\newcommand{\be}{\begin{equation}}
\newcommand{\ee}{\end{equation}}

\newcommand{\om}{\omega}

\theoremstyle{plain}
\newtheorem{thm}{Theorem}
\newtheorem*{thm*}{Theorem}
\newtheorem{prop}[thm]{Proposition}
\newtheorem*{prop*}{Proposition}
\newtheorem{lem}[thm]{Lemma}

\newtheorem{defn}[thm]{Definition}

\theoremstyle{plain}
\newtheorem{rem}{Remark}

\newcommand{\lessim}{\stackrel{<}{\sim}}

\newcommand{\ignore}[1]{}
\newcommand{\av}{-\hspace{-2.4ex}\int}
\newcommand{\del}{\delta\hspace{-0.2ex}}

\title{The magnetization ripple: \\{ a nonlocal stochastic PDE perspective}
}
\author{
Radu Ignat\thanks{Institut de Math\'ematiques de Toulouse, Universit\'e Paul Sabatier, 31062 Toulouse, France (email: Radu.Ignat@math.univ-toulouse.fr)} \and Felix Otto\thanks{Max-Planck-Institut f\"ur Mathematik in den Naturwissenschaften, Inselstr. 22, 04103 Leipzig, Germany (email: Felix.Otto@mis.mpg.de)} 
}

\begin{document}
\maketitle

\begin{abstract}
The magnetization ripple is a microstructure formed by the magnetization in a thin-film
ferromagnet. It is triggered by the random orientation of the grains in the poly-crystalline
material. In an approximation of the micromagnetic model, which is sketched in this paper,
this leads to a nonlocal (and strongly anisotropic) elliptic equation in two dimensions
with white noise as a right hand side. 
However, like in singular Stochastic PDE, this right hand side is too rough for the non-linearity
in the equation. In order to develop a small-date well-posedness theory, we take inspiration from 
the recent rough-path approach to singular SPDE. To this aim, we develop a Schauder theory
for the non-standard symbol $|k_1|^3+k_2^2$.
\end{abstract}
\noindent{\scriptsize\textbf{Keywords:} {singular stochastic PDE, nonlocal elliptic PDE, Schauder regularity, anisotropic H\"older norm, semigroup, micromagnetics.}}\\
{\scriptsize\textbf{MSC:} 35R60,  35J60, 78A30, 82D40. 
}

\tableofcontents

\section{Introduction}

The magnetization ripple in a ferromagnetic thin-film sample
is the response to the polycrystallinity of the sample.
In experiments, it manifests itself as an in-plane oscillation
of the magnetization, predominantly in direction of the main (in-plane) magnetization direction.
The fact that the sample is made up of (comparatively small) randomly oriented grains
leads to an easy axis for the magnetization that is a random field and thus acts like quenched noise.
The {\it anisotropic} response of the magnetization $m$ to this isotropic quenched noise is a consequence
of the non-local interaction of the magnetization given by the stray-field energy.

\medskip

Starting from the three-dimensional micromagnetic (variational) model,
we heuristically derive a reduced model (see Section \ref{Der}) that zooms in on the (different) longitudinal and transversal 
characteristic length scales of the ripple. The reduced model is a two-dimensional, nonlocal variational model formulated in 
terms of the transversal magnetization component $m_2$. 
On these scales, the quenched noise acts like random transversal field
of white-noise character (because the grains are smaller than the characteristic scales).
We argue that this derivation is self-consistent (see also \cite[Section V]{PhyRevB}).

\medskip

There are two main challenges of the derived model that we heuristically point out in Section~\ref{Cha}.
The first challenge can already be seen on
the level of the simplification that gets rid of the anharmonic term. The ensuing linear
Euler-Lagrange equation can be explicitly solved in Fourier space --- however the highest-order
term in the energy, the exchange contribution, diverges (see Section \ref{sec:lin_en} in Appendix). Hence the ripple should
rather be analyzed on the level of the Euler-Lagrange equation than by the direct method
of the calculus of variations.

\medskip
 
The second challenge is more subtle and more serious: The nonlinearity in the Euler-Lagrange
equation is too singular for its right-hand side (RHS) given by white noise $\xi$. More dramatically, one of the
quadratic terms in the Euler-Lagrange equation cannot be given an unambiguous sense even
when one plugs in the solution of the linear Euler-Lagrange equation, which is well-characterized. 
This situation is similar to certain classes of
stochastic partial differential equation (SPDE), i.e., time-dependent nonlinear parabolic equations driven
by space-time white noise. While noise in SPDEs typically models {\it thermal} noise and our noise
is of {\it quenched} nature, and while these SPDEs are typically parabolic and our Euler-Lagrange
equation is of (non-local) elliptic character, the mathematical challenges are identical.

\medskip

In fact, the issue is to make sense of the product of a function and a distribution.
This can be done in an unambiguous sense provided the function is more regular than the distribution
is irregular. In order to make use of this, regularity has to be measured in a way that is consistent with
the (leading-order) linear part of the equation. While in the parabolic case, this requires
spaces that respect the fact that the time derivative is worth two space derivatives, in our
case we have the relationship that two $x_2$-derivates are worth three $x_1$-derivatives.
In the case of {\it stationary} (i.e. shift-invariant) driving noise, like is the case of white noise,
there is no loss in using the scale of H\"older spaces $C^\alpha$ with respect to (w.r.t.) a Carnot-Carath\'eodory
metric that respects the above scaling. On this scale, the crucial product turns out
to be border-line singular: The function in this product is slightly worse than $C^{\frac{3}{4}}$ while
the distribution is slightly worse than $C^{-\frac{3}{4}}$. 

\medskip

This situation is reminiscent of a fundamental problem in the theory of stochastic (ordinary)
differential equations (SDE): The theory requires at a minimum to give a (distributional) sense of the
product of (multi-dimensional) Brownian motion and of its derivative, i.e. white noise. Brownian motion
is known to be slightly worse than $C^{\frac{1}{2}}$ and thus white noise slightly worse than $C^{-\frac{1}{2}}$.
Stochastic analysis has found two ways out of this border-line singular situation: 
Ito calculus and more recently Lyons' rough path theory (see \cite{Lyons}).
While Ito calculus uses the Martingale structure of Brownian motion and thereby the time direction
and thus is not easily amenable to a treatment of irregular spatial noise, rough path theory is oblivious
to this structure. Hairer and coworkers have extended rough path theory from SDEs to 
SPDEs (see \cite{Hairer}). We follow this approach in our -- simpler -- situation.

\medskip

This approach consists of two parts:
The first part consists in giving an off-line definition to the singular product in the 
PDE (for $u$ with RHS given by white noise $\xi$)
when the solution $v$ of the linear (constant-coefficient) equation is plugged in. 
In this case, the singular product $F$ is the product of two {\it Gaussian} fields
and can be characterized by Gaussian calculus: Thanks to stochastic cancellations, an {\it almost-sure} unambiguous 
(distributional) sense can be given to this product that is stable under regularization of white noise by convolution.
This is carried out in Section \ref{Off}.

\medskip

The second part consists in setting up a completely {\it deterministic} (i.e. path-wise)
fixed-point problem in $w:=u-v$ with a RHS given by the distribution $F\in C^{-\frac{3}{4}-\eps}$ for every $\eps>0$. 
All the further non-linearities in the PDE are regular.
For this second part, we have to show that:
\begin{enumerate}
\item[1)] ${\mathcal L}^{-1}C^{\alpha-2}\subset C^\alpha$ for our (anisotropic and nonlocal) linear operator 
${\mathcal L}$ (see \eqref{w07}); 
\item[2)] $C^\alpha \cdot C^\beta\subset C^\beta$ for $\beta<0<\alpha$ with $\alpha+\beta>0$ (i.e. the regular case).
\end{enumerate}
We do both with help of a set of tools recently introduced for SPDEs (see \cite[Section 2]{OttoWeber}).


\section{Derivation of the model}\label{Der}

In this section, we heuristically derive the model we shall analyze. Most
of the arguments can be found in the physics literature \cite{Harte,Hofmann}.
We closely follow the set-up in \cite{PhyRevB} based on \cite{SteinerDiss}.
{The more mathematically oriented readers could skip this section.}

The starting point is the micromagnetic model; in its stationary version, it predicts the magnetization $m$,
which locally indicates the orientation of the elementary magnets on a mesoscopic level, 
of a ferromagnetic sample as the ground state or at least local minimizer of a variational problem.
We present a version that is partially non-dimensionalized in the sense that the magnetization,
the fields and the energy density are non-dimensional, but length is still dimensional. 
The energy $E$ is the sum of the following four terms:
\begin{itemize}
\item The exchange contribution $d^2\int|\nabla m|^2\, \, dx$ models a short-range attractive interaction
of quantum mechanical origin, where the ``exchange length'' $d$ is a material parameter, typically in
the nanometer range. Since we are below the Curie temperature, physically speaking, this term comes
together with a spatially constant, non-vanishing length of the magnetization, which in our
non-dimensionalization turns into the unit-length (and thus non-convex) constraint $|m|^2=1$.
\item The second contribution is the energy $\int|h|^2\, \, dx$  of the stray field $h$, which is determined through $m$ by 
the static Maxwell equations $\nabla\cdot(h+m)=0$ and $\nabla\times h=0$. For later purpose, it is convenient
to think of $h$ as a field to minimize subject to the sole constraint $\nabla\cdot(h+m)=0$.
\item The third contribution is the Zeeman term $-2\int H_{ext}\cdot m \, \, dx$ that models the interaction 
of the magnetization with the external field $H_{ext}$.
\item The fourth term is the crystalline anisotropy $-Q\int (e\cdot m)^2\, \, dx$, where the ``quality factor'' $Q$ is a
non-dimensional material parameter; for $Q>0$ it favors the alignment of the magnetization with
the ``easy axis'' $e$ ($e$ is a unit vector); a ferromagnet is called soft when $Q>0$ is small.
\end{itemize}
We are interested in a sample that comes in form of a thin film of thickness $t$, typically in the
range of tens or hundreds of nanometers. {For simplicity, we think of an infinitely extended ferromagnetic film and thus disregard boundary effects}. Clearly, the exchange, Zeeman, and anisotropy contributions
are restricted to the sample, whereas the stray-field energy still is to be taken over the entire space.
Also the constraint $\nabla\cdot(h+m)=0$ has to be imposed in all of space, with {\bf $m$ extended by zero
outside the sample}, and thus has to be interpreted in the distributional sense:
\begin{align*}
\int_{\mbox{space}}\nabla\zeta\cdot h\, \, dx={-}\int_{\mbox{sample}}\nabla\zeta\cdot m\, \, dx\quad\mbox{for all test functions}
\;\zeta.
\end{align*}
In line with this geometry, we think of an in-plane external field $H_{ext}$. We choose the coordinate
system such that the $x_3$-axis is the thickness direction and the $x_1$-axis the direction of the external
field, which thus assumes the form $H_{ext}=(h_{ext},0,0)$ with $h_{ext}>0$.

\medskip

We are further interested in a polycrystalline sample: The sample is formed by single-crystal grains of a given
easy axis, we assume that the grains' orientation is independently and uniformly distributed, which 
transmits to the easy axis. In other words, $e$ is a random field. In the absence of the anisotropy, the
minimizer would be given by $m=(1,0,0)$ and $h=0$ (recall our simplified setting of an infinitely extended ferromagnetic film) 
\footnote{In order to speak of a global minimizer, it is convenient to pass to a periodic setting in both in-plane directions.} which annihilates exchange, stray field, and Zeeman
contributions. The heterogeneous anisotropy energy creates a torque that perturbs this magnetization.

\medskip

We now start with a couple of model reductions. The thin-film geometry allows for two simplifications:
In conjunction with the exchange energy, one may assume that $m$ only depends on the in-plane
variables $x'=(x_1,x_2)$, i.e., $m=m(x')$. {\bf In the following, the prime $'$ always denotes an in-plane quantity}. In conjunction with the stray-field energy, one may assume that the
$m_3$-component, which generates a surface ``charge'' density at the bottom and top surfaces and thus
generates a stray field, is suppressed, i.e., $m_3=0$. Hence the energy reduces to
\begin{align*}
d^2t&\int|\nabla'm'|^2\, \, dx'+\int|h|^2\, \, dx\nonumber -Qt \int m'\cdot (\av_0^t e'\otimes e'\, \, dx_3) m'\, \, dx'
-2h_{ext}t\int m_1\, \, dx'
\end{align*}
under the constraints $|m'|^2=1$ (where $\displaystyle \av_0^t=\frac1t\int_0^t$) and 
\begin{align*}
\int\nabla\zeta\cdot h\, \, dx=t\int \zeta(\cdot, x_3=0) \, \nabla'\cdot m' \, \, dx'\quad\mbox{for all}\;\zeta.
\end{align*}

We now can already explain the anisotropic response of the magnetization to the heterogeneous anisotropy energy
by considering oscillations of the (in-plane unit-length) magnetization $m'\approx (1,0)$ with wave number $k'$. 
If $k'$ is aligned with the $k_1$-axis, such an oscillation is divergence-free to leading order; in all
other cases, it generates a substantial stray-field and thus is penalized not only by exchange but also 
the stray-field energy. This is indeed confirmed
by experiments: The ripple is predominantly an oscillation in direction of the average magnetization
(and thus helps to visualize the latter in Kerr microscopy). 

\medskip

Based on this discussion, we make the assumption that the typical $x_1$-scale $\ell_1$ of the ripple 
is much smaller than the typical $x_2$-scale $\ell_2$, which will be seen to be self-consistent in a
relevant regime. This assumption means that $\partial_2m'$ is dominated by $\partial_1m'$
and that $h_2$ is more strongly suppressed than $h_1$; hence the exchange energy simplifies to
$d^2 t\int|\partial_1m'|^2\, \, dx'$ and the stray field energy to $\int h_1^2+h_3^2\, \, dx$ under the
constraint 
\begin{align*}
\int\partial_1\zeta h_1+\partial_3\zeta h_3\, \, dx=t\int\zeta(\cdot, x_3=0)\,  \nabla'\cdot m'\, \, dx'\quad\mbox{for all}\;\zeta.
\end{align*}
Note that the $x_2$ variable appears just as a parameter when passing from $\nabla'\cdot m'$ to $h$; taking the Fourier transform in $x_1$
one sees that the stray field energy takes the form of 
$\frac{1}{2}t^2\int(|\partial_1|^{-\frac{1}{2}}\nabla'\cdot m')^2 \, \, dx'$, where the fractional inverse
derivative $|\partial_1|^{-\frac{1}{2}}$ is defined as coming from the Fourier multiplier 
$|k_1|^{-\frac{1}{2}}$. This type of scaling of the stray field energy appears in many studies of thin ferromagnetic films 
\cite{CO1, CO2, COS, DKO, DKMO, IK, IM1, IM2, IO1, IO2, M1, OS} where new mathematical tools are developed in order to understand the structure of domain walls such as N\'eel walls, concertina pattern, Landau state etc. (for more details, see the review paper \cite{DKMO1}).
\medskip

A further reduction is based on the assumption that the amplitude of the
ripple is small in the sense of $|m'-(1,0)|\ll 1$, \footnote{We denote in this section $a\ll b$ if $a/b\to 0$. } 
which also will be seen to be self-consistent. Based on this assumption,
and because of the constraint $|m'|^2=1$, we use $m_1\approx 1-\frac{1}{2}m_2^2$ in the stray-field
and Zeeman contributions and $m_1\approx 1$ in the exchange and anisotropy contributions which leads
to
\begin{align}\label{i1}
d^2t&\int(\partial_1m_2)^2\, \, dx'+\frac{1}{2}t^2\int(|\partial_1|^{-\frac{1}{2}}
(\partial_2m_2-\partial_1\frac{1}{2}m_2^2))^2\, \, dx'\nonumber\\
-2Qt&\int(\av_0^te_1e_2\, \, dx_3)m_2\, \, dx'
+h_{ext}t\int m_2^2\, \, dx'
\end{align}
where we neglected an additive constant independent of the configuration.
From (\ref{i1}) we learn that the crystalline anisotropy acts like the random transversal external field
$h_{ran}(x'):=Q\av_0^t(e_1e_2)(x)\, \, dx_3$ in $x_2$-direction.

\medskip

We finally turn to an (asymptotic) stochastic characterization of this random field $h_{ran}$.
Since the orientation of the grains is uniform, we have
\begin{align}\label{i2}
\langle h_{ran}(x')\rangle=0,
\end{align}
where $\langle\cdot\rangle$ stands for the expectation, or ensemble average in the physics jargon.
Assuming that the (average) size $\ell$ of the grains is not much smaller than the thickness $t$ of
the sample (in the real samples we have in mind, $\ell$ is also in the range of tens of nanometers),
taking the vertical average in $h_{ran}$ does not lead to full cancellation so that we have
\begin{align}\label{i3}
\langle h_{ran}^2(x')\rangle\sim Q^2.
\end{align}
Since the orientation is independent from grain to grain we have
for the covariance (which in view of (\ref{i2}) reduces to a second moment)
\begin{align}\label{i4}
\langle h_{ran}(x')h_{ran}(y')\rangle=0\quad\mbox{for}\;|x'-y'|\gg\ell.
\end{align}
Under the assumption that the grain size is small compared to the typical length scale of the ripple,
that is, under the assumption $\ell\ll\ell_1,\ell_2$, the random field $h_{ran}$ acts on the
magnetization $m_2$ as if (\ref{i3}) and (\ref{i4}) were consolidated into
\begin{align}\label{i5}
\langle h_{ran}(x')h_{ran}(y')\rangle\sim Q^2\ell^2\delta(x'-y'),
\end{align}
where here, $\sim$ means that the Dirac distribution $\delta$ is multiplied by a fixed constant 
that is of order $Q^2\ell^2$, where $\ell^2$ comes from the (average) area of the grains.
Now (\ref{i2}) and (\ref{i5}) means that
\begin{align*}
h_{ran}\approx Q\ell\xi\quad\mbox{with}\quad\xi\;\mbox{white noise};
\end{align*}
{for simplicity, we will set the above constant to unity in the sequel.}
Hence (\ref{i1}) turns into our final form 
\begin{align}\label{i6}
t\Big(d^2&\int(\partial_1m_2)^2\, \, dx'+\frac{1}{2}t\int(|\partial_1|^{-\frac{1}{2}}
(\partial_2m_2-\partial_1\frac{1}{2}m_2^2))^2\, \, dx'\nonumber\\
-2Q\ell&\int\xi m_2\, \, dx'+h_{ext}\int m_2^2\, \, dx'\Big).
\end{align}

\medskip

We now pass to suitable reduced units in (\ref{i6}) (which also amounts to
a non-dimensionalization of length) which has the merit of: 
\begin{enumerate}
\item[1)] getting rid of the
various parameters in (\ref{i6});  
\item[2)] revealing in which regime our assumptions
$\ell\ll\ell_1\ll\ell_2$ and $|m_2|\ll 1$ are self-consistent.
\end{enumerate}
We start with length and, in line with our interpretation of $\ell_1$ and $\ell_2$, we make
the (anisotropic) Ansatz $x_1=\ell_1\hat x_1$ and $x_2=\ell_2\hat x_2$. We'd like to choose
$\ell_1$ and $\ell_2$ such that the energy densities of exchange, Zeeman, and the harmonic part
of the stray field contributions, that is,
\begin{align*}
d^2(\partial_1m_2)^2,\quad\frac{{t}}{2}(|\partial_1|^{-\frac{1}{2}}\partial_2m_2)^2,\quad h_{ext}m_2^2,
\end{align*}
balance. This is achieved for
\begin{align*}
\ell_1=h_{ext}^{-\frac{1}{2}}d,\quad \ell_2=h_{ext}^{-\frac{3}{4}}(t\, d)^\frac{1}{2},
\end{align*}
which is consistent with $\ell\ll\ell_1\ll\ell_2$ provided the stabilizing external field $h_{ext}$
that sets the predominant direction is sufficiently small. We now turn to the reduced transverse
magnetization $m_2$, which we choose such as to balance the two terms in the stray-field energy
$\partial_2m_2-\partial_1\frac{1}{2}m_2^2$ 
$=\ell_2^{-1}\hat\partial_2m_2-\ell_1^{-1}\hat\partial_1\frac{1}{2}m_2^2$, which is achieved for
$m_2=\frac{\ell_1}{\ell_2}\hat m_2$.  Again, this is self-consistent with our small-amplitude
assumption as soon as $\ell_1\ll \ell_2$. Taking 
$h_{ext}^\frac{3}{2}\frac{d}{t}$ as the reduced unit for the energy (area) density, we end up with
\begin{align}\label{i8}
&\int(\hat\partial_1\hat m_2)^2d\hat x'+\int(|\hat\partial_1|^{-\frac{1}{2}}
(\hat\partial_2\hat m_2-\hat\partial_1\frac{1}{2}\hat m_2^2))^2d\hat x'
-2\sigma\int\hat\xi \hat m_2d\hat x'+\int \hat m_2^2d\hat x'
\end{align}
with $\sigma:={h_{ext}^{-5/8}}Qd^{-\frac{5}{4}}\ell t^\frac{1}{4}$ is the renormalized strength of the
transverse field of white-noise character $\hat\xi$. Here we used that the distribution of
white noise is invariant under $\xi=\frac{1}{\sqrt{\ell_1\ell_2}}\hat\xi$ (since the Dirac distribution
that characterizes its covariance scales with $\frac{1}{\ell_1\ell_2}$). We note that both conditions $\ell\ll\ell_1\ll\ell_2$ and $\sigma\ll 1$ are satisfied
for a wide range of stabilizing external fields for typical material parameters. \footnote{{For Permalloy thin films, we have $d=5nm$, $Q=2.5\times 10^{-4}$ with a typical thickness
of $t=100nm$ and a grain size $\ell=20nm$, so that the conditions $\ell\ll\ell_1\ll\ell_2$ and $\sigma\ll 1$ are equivalent with a choice of
$Q^{8/5}(\ell/d)^{8/5}(t/d)^{2/5}\sim 5\times 10^{-5}\ll h_{ext} \ll 6\times 10^{-2}\sim \min\{(t/d)^2, (d/\ell)^2\}$. }} (The model \eqref{i8} was rigorously deduced via 
$\Gamma$-convergence in \cite{COS}, see also \cite{OS, PhyRevB}).

\medskip

\nd {\bf Our reduced (non-dimensionalized) model}. We drop the hats in (\ref{i8}) and make one last, purely mathematically motivated, simplification.
As we shall see, the main challenge in the model lies in the fact that the white noise 
triggers {\it small} scales of $m_2$. The zero-order term $\int m_2^2\, \, dx'$ coming from
the Zeeman contribution does not much affect the small scales. Rather,
this term penalizes scales of $m_2$ much larger than one. 
For mathematical convenience, we replace this mechanism 
by another mechanism of the same effect, namely periodic boundary conditions with period one, that is, 
\begin{align}\label{i9}
m_2(x_1+1,x_2)=m_2(x_1,x_2+1)=m_2(x_1,x_2),
\end{align}
and drop the last term
\begin{align}\label{i10}
\int_{[0,1)^2}(\partial_1 m_2)^2\, \, dx'+\int_{[0,1)^2}(|\partial_1|^{-\frac{1}{2}}
(\partial_2m_2-\partial_1\frac{1}{2} m_2^2))^2\, \, dx'-2\sigma\int_{[0,1)^2}\xi m_2\, \, dx'.
\end{align}
This being said,
our results would remain valid, and the proofs become only slightly more involved, when keeping the
zero-order term $\int m_2^2\, \, dx'$ alongside the periodic boundary conditions. However, it would require
new arguments to get rid of the (artificial) periodic boundary conditions. This is a well-known
effect when dealing with (quenched or thermal) noise: even if scaled by a small constant, it will
be almost surely somewhere too large in the infinite plane for our arguments.

\medskip

The periodic boundary conditions (\ref{i9}) in conjunction with the (reduced) stray field
contribution in (\ref{i10}) lead to a new constraint. The stray field contribution is only
finite provided the expression under the inverse fractional operator
$|\partial_1|^{-\frac{1}{2}}$ has vanishing average in $x_1$ for all $x_2$. Because of the
periodic boundary conditions, the second contribution $\partial_1\frac{1}{2}m_2^2$ has
vanishing average in $x_1$ so that we need $\partial_2m_2$ to have vanishing average in $x_1$,
which means that $\int_0^1m_2\, \, dx_1$ does not depend in $x_2$. We impose something slightly stronger, namely
\begin{align*}
\int_0^1m_2\, \, dx_1=0\quad\mbox{for all}\;x_2.
\end{align*}

\smallskip

\section{Main results}
\label{Cha}

\smallskip

For the sake of a simple notation, we replace $x'=(x_1,x_2)$ by $x$ and $m_2$ by $u$, so that the configurations $u$ 
are $1$-periodic functions in both variables and $\int_0^{1} u(x_1,x_2)\, dx_1=0$ for all $x_2\in (0,1)$. 
Besides the Schwartz test functions (that are defined on the full space $\R^2$),
all other functions and distributions are periodic w.r.t. the two-dimensional torus $[0,1)^2$.

We will focus on the (formal) Euler-Lagrange equation of functional \eqref{i10}:
\begin{align}\label{w01}
(-\partial_1^2-|\partial_1|^{-1}\partial_2^2)u+P(u\partial_2Ru)+\frac{1}{2}\partial_2 Ru^2-\frac{1}{2}P(u\partial_1Ru^2)=\sigma P\xi,
\end{align}
with (periodic) white noise $\xi$ and a small constant $\sigma>0$, where $P$ is the $L^2$-orthogonal projection onto the set of functions of vanishing 
average in $x_1$ (extended in the natural way 
to periodic distributions) and $R$ is the Hilbert transform acting on $1$-periodic functions $f$ in $x_1$ direction as
$$R:=\frac{\partial_1}{|\partial_1|}, \quad \textrm{i.e.,} \quad Rf(k_1)=\begin{cases} i \sgn(k_1) f(k_1) & \quad k_1\in 2\pi \Z\setminus \{0\},\\
0 & \quad k_1=0,
\end{cases}
$$
where the Fourier coefficients of $f$ are $f(k_1)=\int_0^1 e^{-ik_1x_1} f(x_1)\, dx_1$ for $k_1\in 2\pi \Z$ and $\sgn$ is the signum function. In particular, $RP=PR=R$.

The functional framework is given by anisotropic H\"older spaces. More precisely, the leading-order operator $-\partial_1^2-|\partial_1|^{-1}\partial_2^2$ in \eqref{w01} suggests to endow space ${[0,1)^2}$ with
a (Carnot-Carath\'eodory~-~)~metric that is homogeneous w.r.t. the scaling 
$(x_1,x_2)=(\ell \hat x_1,\ell^\frac{3}{2}\hat x_2)$. The simplest expression is given by
\begin{align*}
d(x,y):=|x_1-y_1|+|x_2-y_2|^\frac{2}{3},\quad x, y\in {[0,1)^2},
\end{align*}
which in particular means that we take the $x_1$-variable as a reference.
We now introduce the scale of H\"older semi-norms based on the distance function $d$,
where we restrict ourselves to the range $\alpha\in(0,\frac{3}{2})$ needed in this work. 

\medskip

\begin{defn}\label{definitionCalpha}
For a {periodic} function $f$, we denote by $\|f\|=\sup_{x} |f(x)|$ the supremum norm of $f$. For an exponent $\alpha\in(0,\frac{3}{2})$, we define
$$
{[f]_\alpha}
:=\sup_{x\not=y}\frac{1}{d^\alpha(y,x)}
\left\{\begin{array}{lc}
|f(y)-f(x)|&\mbox{for}\;\alpha\in(0,1],\\
|f(y)-f(x)-\partial_1f(x)(y_1-x_1)|&\mbox{for}\;\alpha\in(1,\frac{3}{2})
\end{array}\right\}.
$$
We denote by $C^\alpha$ the space of {periodic} functions $f$ with $[f]_\alpha<\infty$.
\end{defn}

\medskip

Our main result, Theorem \ref{T1}, starts from the Euler-Lagrange equation \eqref{w01} with periodic white noise $\xi$ replaced by its convolution $\xi_\ell:=\phi_\ell*\xi$, where 
\begin{align}\label{f37}
\phi_\ell(x_1,x_2):=\frac{1}{\ell^{\frac{5}{2}}}\phi(\frac{x_1}{\ell},\frac{x_2}{\ell^\frac{3}{2}}), \quad x\in \R^2,
\end{align}
and $\phi$ is some symmetric Schwartz function with $\int_{\R^2}\phi\, dx=1$. 
This approximation is natural in view of the heuristic derivation of the equation.
Provided $\sigma>0$ is sufficiently small, Theorem \ref{T1} ascertains a small solution $u^\ell$ and monitors its
distance to the solution of the linear problem. Moreover, the latter is given by $\sigma v_\ell$,
where $v_\ell:=\phi_\ell*v$ is the mollification of the solution $v$ {of vanishing average in $x_1$} of the linearized equation:
\be
\label{w07}
\LL v:=(-\partial_1^2-|\partial_1|^{-1}\partial_2^2)v=P\xi\quad\mbox{in a distributional sense}.
\ee

\begin{thm}\label{T1}
Fix an $\eps\in(0,\frac{1}{4})$, which we think of being small. 
Let $\xi$ be distributed like white noise on the torus $[0,1)^2$ under the expectation $\langle\cdot\rangle$.
Then there exists a deterministic constant $C<\infty$ and a random constant $\sigma_0>0$
with $\langle\sigma_0^{-p}\rangle<\infty$ for all $1\leq p<\infty$ which is a threshold
in the following sense: Provided $\sigma\in[0,\sigma_0]$ and for every $0<\ell\le 1$, there exists a unique smooth and periodic $u^\ell$ 
of vanishing average in $x_1$ such that 
\begin{align*}
\LL u^\ell
+P(u^\ell \partial_2 Ru^\ell)+\frac{1}{2}\partial_2R(u^\ell)^2-\frac{1}{2}P(u^\ell \partial_1R(u^\ell)^2)=\sigma P\xi_\ell
\end{align*}
and that is small enough in the sense of
\begin{align*}
[u^\ell-\sigma v_\ell]_{\frac{5}{4}-\eps}\le C\left(\frac\sigma{\sigma_0}\right)^2.
\end{align*}
Moreover, as $\ell\downarrow 0$, $u^\ell$ converges {in $C^{3/4-\eps}$} to a limit that is independent of the choice of the 
(symmetric) convolution kernel $\phi$.
\end{thm}

The last sentence of Theorem \ref{T1} is the most important: The limit of (small) solutions with regularized noise
is independent of the regularization. In fact, the limit can be characterized as a solution of the 
limiting Euler-Lagrange equation as we shall do in Theorem \ref{T2}. However, giving a rigorous meaning to this equation requires some effort
as we shall explain now, also putting light behind the exponents $\frac{3}{4}$ and $\frac{5}{4}$. 
As is usual, the amount of irregularity of white noise as a distribution depends on the space dimension; the rule of
thumb is that the order is just below $-\frac{1}{2}\times\mbox{effective dimension}$. In view of our anisotropic
metric $d$, the effective dimension is $1+\frac{3}{2}=\frac{5}{2}$. Hence we expect $\xi$ to be a distribution
of order just below $-\frac{5}{4}$, which is indeed true on the level of the following H\"older spaces of negative exponents;  
we will restrict to the range we require in this
work, namely $\beta\in(-\frac{3}{2},0)$.

\smallskip

\begin{defn}\label{definitionCbeta}
Let $f$ be a periodic distribution. In case of $\beta\in(-1,0)$ we set
\begin{align*}
[f]_\beta:=\inf\{|c|+[g]_{\beta+1}+[h]_{\beta+\frac{3}{2}}\, :\, f=c+\partial_1g+\partial_2h\}
\end{align*}
and in case of $\beta\in(-\frac{3}{2},-1]$ we set
\begin{align*}
[f]_\beta:=\inf\{|c|+[g]_{\beta+2}+[h]_{\beta+\frac{3}{2}}\, :\, f=c+\partial_1^2g+\partial_2h\}.
\end{align*}
\end{defn}
In both cases, the expressions are interpreted as $+\infty$ if the distribution $f$ does not allow
for a representation in terms of two periodic functions $g$ and $h$ and a constant $c$. We denote by $C^\beta$ the space of periodic distributions $f$ with $[f]_\beta<\infty$. 
We now state the regularity of the white noise:

\smallskip

\begin{lem}\label{L1}
We have for all $1\leq p<\infty$ and $0<\eps<\frac14$, \footnote{Here and in the sequel, $a \lessim b$ means $a\leq Cb$ with a generic constant $C>0$ that depends on the exponents in the statement of the respective result, e.g., $p$ and $\eps$ (but not $\ell_0$) in case of Lemma \ref{L1}. }
\begin{align*}
 \langle\sup_{0<\ell\le 1}[P\xi_\ell]_{-\frac{5}{4}-\eps}^p\rangle\lesssim 1\quad\mbox{and}\quad
  \langle\sup_{0<\ell\le\ell_0}[\xi_\ell-\xi]_{-\frac{5}{4}-\eps}^p\rangle^{1/p}\lesssim {\ell_0^{\eps/2}} \quad \textrm{ for } \ell_0\leq 1.
\end{align*}
\end{lem}

The above estimate holds also true for the white noise $\xi$ (instead of the projected distribution $P\xi$); since we only need $P\xi$ in the sequel, 
we will restrict to the estimate in Lemma \ref{L1}.

Another rule of thumb is that the elliptic operator ${\cal L}=-\partial_1^2-|\partial_1|^{-1}\partial_2^2$ increases
regularity by two increments (in our anisotropic metric where the first component in the ``numeraire'' or unit of reference).
While this does not fall into the realm of standard Schauder theory because of the non-locality of the elliptic
operator, it is indeed true on our H\"older scale:

\begin{lem}\label{L2}
Let $\alpha\in({\frac12},\frac{3}{2}){\setminus\{1\} }$. For any periodic function $f$ with vanishing average in $x_1$ we have

\begin{align*}
{[Rf]_\alpha}, [f]_\alpha\lesssim[(-\partial_1^2-|\partial_1|^{-1}\partial_2^2)f]_{\alpha-2}.
\end{align*}
\end{lem}

Hence for the solution $v$ of the linearized equation (\ref{w07}), we obtain from Lemmas \ref{L1} and \ref{L2} that
\begin{align}\label{w07bis}
\langle\sup_{0<\ell\le 1}[v_\ell]_{\frac{3}{4}-\eps}^p\rangle<\infty\quad\mbox{for all}\;1\leq p<\infty.
\end{align}
In particular, we have almost surely $[v]_{\frac{3}{4}-\eps}<\infty$, but expect no regularity beyond
$\frac{3}{4}$. This is reminiscent of Brownian motion, which almost surely is H\"older continuous
with exponent $\frac{1}{2}-\eps$, but almost surely not H\"older continuous with exponent $\frac{1}{2}$ 
(even not on small intervals). We are now in the position to explain the difficulty with (\ref{w01}):
One of the three non-linear terms cannot be given a sense when $u$ is substituted by $v$ (and thus
we expect the same problem for $u$ itself). At first sight, both non-divergence form terms
\begin{align}\label{w02}
v\partial_2Rv,\quad \frac{1}{2}v\partial_1Rv^2
\end{align}
look difficult, since neither $\partial_2{R}v$ nor even $\partial_1 {R}v$ exist classically. However,
like is trivially the case for the third term $\frac{1}{2} \partial_2 Ru^2$,
one of the terms in (\ref{w02}) can be given a distributional sense thanks to the following result on the
product of a function $u$ and a distribution $f$. Loosely speaking, $uf$ can be given a canonical sense
as a distribution (not better than $f$) if $u$ is more regular than $f$ is irregular. Results like these are classical and we need
a variant compatible with the anisotropic scaling dictated by the anisotropic metric $d$. 

\smallskip

\begin{lem}\label{L3}
Let $\alpha\in (0,\frac{3}{2})$ and $\beta\in(-\frac{3}{2},0){\setminus \{-1, -\frac12\}}$
with $\alpha+\beta>0$. If $u$ is a periodic function with $[u]_\alpha<\infty$ and
$f$ is a periodic distribution with $[f]_\beta<\infty$, then there exists a distribution denoted by
$u f$ such that for all convolution scales $\ell\le 1$: 
\begin{align}\label{w23}
\ell^{\alpha+\beta}[u]_\alpha[f]_\beta \gtrsim
\begin{cases} \|\lceil u, (\cdot)_\ell\rceil f\| &\quad \textrm{ if } \alpha \in (0,1],\\
\| \big(  \lceil u, (\cdot)_\ell\rceil-\partial_1 u \lceil x_1, (\cdot)_\ell\rceil \big) f\| &\quad \textrm{ if } \alpha \in (1, \frac32),
\end{cases}
\end{align}
where we denoted the commutator-convolution:
$$\lceil g, (\cdot)_\ell\rceil f:=gf_\ell-(gf)_\ell$$ for the function $g\in \{u, x_1\}$. \footnote{If $g(x)=x_1$, then $gf$ has the (standard) meaning of a product between a $C^\infty$ function and a distribution.}
This property characterizes $u f$ uniquely, even independently of the (Schwartz symmetric) convolution kernel $\phi$.
Moreover, provided $u$ has vanishing average in $x_1$,
\begin{align}\label{w03}
[u f]_\beta\lesssim[u]_\alpha[f]_\beta.
\end{align}
\end{lem}

Finally, while the usual Hilbert transform is bounded on H\"older spaces, our one-dimensional
Hilbert transform $R$ is not; there is a logarithmic loss in the order:

\smallskip

\begin{lem}\label{L4}
Let $\alpha$ and $\eps>0$ with $\alpha,\alpha-\eps\in{(0, \frac{3}{2})}$. 
Then for any function $f$ of vanishing average in $x_1$:
\begin{align}\label{w04}
[Rf]_{\alpha-\eps}\lesssim[f]_\alpha.
\end{align}
\end{lem}

From Lemma \ref{L3} and Lemma \ref{L4} we learn that the cubic term in (\ref{w02}) poses no
fundamental problem  (because of $(\frac{3}{4}-\eps)-(\frac{1}{4}-2\eps)>0$
for $\eps\ll 1$). It can be given a sense and estimated as a distribution of order 
slightly below $-\frac{1}{4}$:
\begin{align*}
[\frac{1}{2}v \partial_1 R v^2]_{-\frac{1}{4}-2\eps}
&\stackrel{(\ref{w03})}{\lesssim} [v]_{\frac{3}{4}-\eps}[\partial_1Rv^2]_{-\frac{1}{4}-2\eps}
\stackrel{Def\, \ref{definitionCbeta}}{\le} [v]_{\frac{3}{4}-\eps}[Rv^2]_{\frac{3}{4}-2\eps}
\stackrel{(\ref{w04})}{\lesssim} [v]_{\frac{3}{4}-\eps}[v^2]_{\frac{3}{4}-\eps}
{\lesssim} [v]_{\frac{3}{4}-\eps}^3
\end{align*}
and thus by (\ref{w07bis})
\begin{align*}
\langle[\frac{1}{2} v \partial_1 Rv^2]_{-\frac{1}{4}-2\eps}^p\rangle<\infty\quad\mbox{for all}\;p<\infty.
\end{align*}
The real issue comes from the quadratic term in (\ref{w02}): Since 
we expect $v$ to have regularity $\alpha$ slightly below $\frac{3}{4}$ and thus,
$\partial_2 v$ (also, $\partial_2 Rv$) to have regularity $\beta$ slightly below $\frac{3}{4}-\frac{3}{2}=-\frac{3}{4}$,
we just miss the condition $\alpha+\beta>0$ required by Lemma \ref{L3}. Hence
we need an ``off-line'' stochastic treatment of the term $v_\ell\partial_2Rv_\ell$. 

\begin{lem}\label{L5}
Consider $F^\ell:={P}(v_\ell\partial_2Rv_\ell)$. 
We have for all $1\leq p<\infty$ and $0<\eps<\frac14$:
\begin{align*}
 \langle\sup_{\ell\le 1}[F^\ell]_{-\frac{3}{4}-\eps}^p\rangle&\lesssim 1,\\
\langle\sup_{\ell,\ell'\le \ell_0}[F^\ell-F^{\ell'}]_{-\frac{3}{4}-\eps}^p\rangle^{1/p}&\lesssim { \ell_0^{\eps/2}} \quad \textrm{ for } \ell_0\leq  1.
\end{align*}
In particular, almost surely, $\{F^\ell\}_{\ell\downarrow0}$ is a Cauchy ``sequence'' in the Banach space defined through $[\cdot]_{-\frac{3}{4}-\eps}$ and thus has a limit $F$ such that $F=PF$. Moreover, almost surely,
$F$ does not depend on the (Schwartz symmetric) convolution kernel $\phi$.
\end{lem}

Equipped with $v$ and $F$, we now may characterize the limit $\lim_{\ell\downarrow 0}u^\ell$ in Theorem \ref{T1}.
To this purpose, we (formally) rewrite (\ref{w01}) in terms of $w=u-\sigma v$ and substitute
${P}(v\partial_2Rv)$ by $F$.

\begin{thm}\label{T2}
This is a continuation of Theorem \ref{T1}. We have 
\begin{align}
\label{fi18}
\lim_{\ell\downarrow 0}[u^\ell-\sigma v_\ell-w]_{\frac{5}{4}-\eps}=0,\quad
\lim_{\ell\downarrow 0}[v_\ell-v              ]_{\frac{3}{4}-\eps}=0,
\end{align}
where $w$ is the unique periodic function with vanishing average in $x_1$ that satisfies (in a distributional sense)
\begin{align}
\lefteqn{(-\partial_1^2-|\partial_1|^{-1}\partial_2^2)w}\nonumber\\
&+P\big(\sigma^2F+\sigma v\partial_2 Rw+\sigma w\partial_2 Rv+w\partial_2Rw\big)\nonumber\\
&+\frac{1}{2} \partial_2 R(w+\sigma v)^2-
\frac{1}{2} P((w+\sigma v) \partial_1R(w+\sigma v)^2)=0\label{w08}
\end{align}
and that is small in the sense of
\begin{align}\label{w09}
[w]_{\frac{5}{4}-\eps}\le C\left(\frac\sigma{\sigma_0}\right)^2.
\end{align}
\end{thm}

Let us comment on the distributional interpretation of the non-linear terms in the equation (\ref{w08}):
There is no issue with $\frac{1}{2} \partial_2R(w+\sigma v)^2$ and of course none with $\sigma^2 PF$.
Based on Lemmas~\ref{L3} and \ref{L4}, we have given the argument for 
the cubic term $\frac{1}{2} P((w+\sigma v) \partial_1R(w+\sigma v)^2)$, which just relied
on $[w+\sigma v]_{\frac{3}{4}-\eps}<\infty$. The three terms
$v\partial_2 Rw$, $w\partial_2 Rv$, and (in particular) $w\partial_2 Rw$ are regular in the sense of Lemma \ref{L3}:
$v$ has regularity $\frac{3}{4}-\eps$, $w$ has regularity $\frac{5}{4}-\eps$ and $\partial_2$
reduces the regularity by $\frac{3}{2}$, which still gives a positive sum $\frac{3}{4}-\eps+\frac{5}{4}-\eps-\frac{3}{2}>0$
for $\eps\ll 1$. In fact, we can motivate the exponent $\frac{5}{4}-\eps$ in (\ref{w09}) as follows:
The worst distributions in the two last lines of (\ref{w08}) are of the order
$-\frac{3}{4}-\eps$; hence by Lemma \ref{L2}, $w$ is expected to be of order $-\frac{3}{4}-\eps+2$
$=\frac{5}{4}-\eps$.

\medskip

We will establish Theorems \ref{T1} and \ref{T2} by formulating (\ref{w08}) as a
fixed point problem in the ball described by (\ref{w09}). Alongside the limiting
fixed point problem, we will also consider the one where $F$ and $v$ in (\ref{w08}) are replaced
by $F^\ell$ and $v_\ell$, respectively. The convergence of the fixed points will then
follow from the convergences in Lemmas \ref{L1} and \ref{L5}.

{\it Outline of the paper}. In Section \ref{sec:holder_space}, we characterize 
the H\"older spaces $C^\beta$, $\beta<0$
introduced in Definition \ref{definitionCbeta}; more precisely, we give an equivalent norm for distributions in $C^\beta$ with $\beta\in (-\frac32, 0)\setminus\{-1, -\frac12\}$ in Lemma \ref{Lequiv} and we prove Lemmas \ref{L2}, \ref{L3} and \ref{L4}. In Section \ref{sec:main}, we prove our main results in Theorems \ref{T1} and \ref{T2}. In Section \ref{sec:sto}, we prove the estimates of the stochastic terms in Lemmas \ref{L1} and \ref{L5}. Finally, in the appendix, we prove in Section \ref{sec:lin_en} that the linearized energy functional does not admit (with a positive probability) critical points of finite energy, while in Section \ref{sec:stand} we recall some standard results for the anisotropic H\"older spaces.


\section{Anisotropic H\"older spaces. Proof of Lemmas \ref{L2}, \ref{L3}, and \ref{L4}}
\label{sec:holder_space}

The proof of Lemma \ref{L3} on products of functions and distributions relies
on an equivalent characterization of the {H\"older norms for a negative exponent} defined in Definition
\ref{definitionCbeta}, which is stated in Lemma \ref{Lequiv} below and the proof of which
relies on Schauder theory for the nonlocal elliptic operator
$${\mathcal A}:=|\partial_1|^3-\partial_2^2,$$ an operator
which clearly is in line with the scaling properties of the distance $d$.
Hence we prove Lemma \ref{L2} on Schauder theory for 
$${\cal L}=-\partial_1^2-|\partial_1|^{-1}\partial_2=|\partial_1|^{-1}{\mathcal A}$$ alongside with 
Lemma \ref{Lequiv}. The equivalent characterization of the negative {exponent} H\"older
norms relies on the ``heat kernel'' of ${\mathcal A}$ used as a convolution family; 
the Fourier transform of $\{\psi_T\}_{T>0}$ is evidently given by
\begin{align}\label{eq14bis}
\psi_T(k)=\exp(-T(|k_1|^3+k_2^2)), \quad \forall k\in \R^2,
\end{align}
and has scaling properties in line with $d$, namely
\begin{align}
\label{f25}
\psi_T(x_1,x_2)=\frac{1}{(T^{1/3})^{1+\frac{3}{2}}}
\psi(\frac{x_1}{T^{1/3}},\frac{x_2}{(T^{1/3})^\frac{3}{2}}), \quad \forall x\in \R^2,
\end{align}
where for simplicity we write $\psi:=\psi_1$. For a periodic distribution $f$, we denote by
$f_T$ its convolution with $\psi_T$, i.e., $f_T=\psi_T*f$, which yields a smooth periodic function; the semi-group property
\begin{align}\label{f24}
(f_t)_T=f_{t+T}\quad\mbox{for all}\;t,T>0
\end{align}
will be very convenient. 

Before embarking on the proofs, a remark on periodic distributions is in place (``periodic'' always
means periodic of period 1 in the two variables $x_1$ and $x_2$). By the space of periodic distributions $f$
we understand the (topological) dual of the space of $C^\infty$ functions $u$ on the torus (endowed with the
family of semi norms $\{\|\partial_1^j\partial_2^\ell u\|\}_{j,\ell\ge 0}$). As such, the spatial average
$\int_{[0,1)^2}f$ and, more generally, the Fourier coefficients $f(k)$ $=\int_{[0,1)^2}\exp({ -} ik\cdot x)f(x)dx$ for $k\in(2\pi\mathbb{Z})^2$ are well-defined. For a $C^\infty$-function $\psi$ with integrable derivatives,
i.\ e.\ $\int_{\mathbb{R}^2}|\partial_1^j\partial_2^\ell\psi|$ $<\infty$ for all $j,k\ge 0$, we can also
give a sense to $\int_{\mathbb{R}^2} f \psi$ that is consistent with the classical case, 
and which is needed to give a sense to the convolution $\psi*f$ as a periodic $C^\infty$ function
(``convolution'' always means convolution on $\mathbb{R}^2$): 
Indeed, for such functions $\psi$, the periodization $u:=\sum_{z\in\mathbb{Z}^2}\psi(\cdot-z)$ is
well-defined and in $C^\infty$, so that we may set $\int_{\mathbb{R}^2} f \psi$ 
$:=\int_{[0,1)^2}f u$. We note that this definition implies on the level of Fourier coefficients
$(\psi*f)(k)=\psi(k) f(k)$ for all $k\in(2\pi\mathbb{Z})^2$, 
where $\psi(k):=\int_{\mathbb{R}^2}\exp({-} ik\cdot x)\psi(x)dx$ is the Fourier transform of $\psi$.
Indeed, with the above periodization $u$ of $\psi$ we have $(\psi*f)(k)$ $=u(k)f(k)$, so that
the statement reduces to the elementary relation $u(k)=\psi(k)$ between the Fourier series of the periodization $u$
and the Fourier transform of the original function $\psi$.

Note that if $f$ is a periodic distribution with $[f]_\beta<\infty$ for some $\beta\in (-\frac32, 0)$, then the constant $c$ in the decomposition of $f$ in Definition \ref{definitionCbeta} is unique and represents the average $\int_{[0,1)^2}f\, dx$ of the periodic distribution $f$.
Therefore, $[f]_\beta=\left|\int_{[0,1)^2}f\, dx \right|+\left[f-\int_{[0,1)^2}f\, dx\right]_\beta$ if $\beta\in (-3/2, 0)$.

\medskip

\subsection{An equivalent $C^\beta$-norm, $\beta<0$. Proof of Lemma \ref{L2}}

We are now in the position to state Lemma \ref{Lequiv}:
\medskip

\begin{lem}\label{Lequiv}
Let $f$ be a periodic distribution in ${[0,1)^2}$.
\footnote{In the case of periodic distributions $f$ {\it of vanishing average on $[0,1)^2$}, one can consider the $\sup$ over all $T>0$ in \eqref{eq25} (respectively, over all $\ell>0$ in \eqref{supl_numa}).}

i) For $\beta\in(-\frac{3}{2},-1)\cup(-1,-\frac{1}{2})\cup(-\frac{1}{2},0)$, we have
\begin{align}\label{eq25}
[f]_\beta\sim\sup_{T\in(0,1]} (T^{1/3})^{-\beta}\|f_T\|,
\end{align}
where we recall that $\|\cdot\|$ denotes the supremum norm, while $a\sim b$ means $a\lesssim b$ and $b\lesssim a$.

ii) For $\beta\in (-\frac 32, 0)$, 
then 
\be
\label{supl_numa}
\sup_{\ell\in (0,1]} \ell^{-\beta} \|f_\ell\| \lesssim [f]_\beta.
\ee
\end{lem}

Compared to \cite{OttoWeber}, where similar tools are used, the main difference
is that the mask $\psi=\psi(x)$ defining the semi-group convolution family is not a Schwartz function:
While being smooth, it only has mild decay due to the limited smoothness of $\psi(k)$ in $k_1=0$, cf. \eqref{eq14bis}.

\begin{proof}[Proofs of Lemmas \ref{L2} and \ref{Lequiv}]

\newcounter{L2} 
\refstepcounter{L2} 

{\it Step} \arabic{L2}.\label{moments}\refstepcounter{L2} 
{\it Moment bounds on the kernel.} We claim that for all orders of derivative $j,l\ge 0$ and
exponents $\alpha\ge 0$ 
\begin{align*}
\int_{\mathbb{R}^2}|\partial_1^j\partial_2^l\psi|d(x,0)^\alpha dx<\infty\quad\mbox{provided}\;\alpha\le j+2,\\
\int_{\mathbb{R}^2}\big|\partial_1^j\partial_2^l|\partial_1|\psi\big|d(x,0)^\alpha dx<\infty\quad\mbox{provided}\;\alpha\le j.
\end{align*}
In view of its definition on the Fourier level $\psi(k)=\exp(-|k_1|^3-k_2^2)$, cf. (\ref{eq14bis}),
the kernel tensorizes into a Gaussian in $x_2$ and a kernel $\varphi(x_1)$. Hence the above statements reduce to
\begin{align*}
\int_{\mathbb{R}}|\partial_1^j\varphi|(|x_1|+1)^{j+2} dx_1<\infty\quad\mbox{and}\quad
\int_{\mathbb{R}}\big|\partial_1^j|\partial_1|\varphi\big|(|x_1|+1)^j dx_1<\infty.
\end{align*}
By Cauchy-Schwarz and $\int_{\mathbb{R}}(|x_1|+1)^{-2}dx_1<\infty$, these statements in turn reduce to
\begin{align*}
\int_{\mathbb{R}}|\partial_1^j\varphi|^2(|x_1|+1)^{2(j+3)} dx_1,\;\int_{\mathbb{R}}\big|\partial_1^j|\partial_1|\varphi\big|^2(|x_1|+1)^{2(j+1)} dx<\infty.
\end{align*}
By Plancherel, this can be expressed as
\begin{align*}
\int_{\mathbb{R}}\big|\partial_{k_1}^{j+3}(k_1^j\varphi)\big|^2dk_1,\;
\int_{\mathbb{R}}\big|\partial_{k_1}^{j+1}(k_1^j|k_1|\varphi)\big|^2 dk_1<\infty.
\end{align*}
These statements hold since near $k_1=0$, $k_1^j\varphi=k_1^j\exp(-|k_1|^3)$ has a bounded $(j+3)$-th derivative 
and $k_1^j|k_1|\varphi=k_1^j|k_1|\exp(-|k_1|^3)$ has a bounded $(j+1)$-th derivative.

\medskip

{\it Step} \arabic{L2}.\label{scaling}\refstepcounter{L2}
{\it Scaling}. We claim that for all orders of derivative $j,l\ge 0$,
exponents $\alpha\ge0$ and convolution parameters $T>0$,
\begin{align}
\int_{\mathbb{R}^2}|\partial_1^j\partial_2^l\psi_T|d(x,0)^\alpha dx
&\lesssim(T^{1/3})^{-j-\frac{3}{2}l+\alpha}, \, \quad\mbox{provided}\;\alpha\le j+2,\label{e03}\\
\int_{\mathbb{R}^2}\big|\partial_1^j\partial_2^l|\partial_1|\psi_T\big|d(x,0)^\alpha dx
&\lesssim(T^{1/3})^{-j-\frac{3}{2}l-1+\alpha},  \, \quad\mbox{provided}\;\alpha\le j.
\end{align}
This follows from Step \ref{moments} via the (anisotropic) change of variables 
$x_1=T^{1/3}\hat x_1$, $x_2=(T^{1/3})^\frac{3}{2}\hat x_2$,
which of course implies $\partial_1=(T^{1/3})^{-1}\hat\partial_1$ 
and $\partial_2=(T^{1/3})^{-\frac{3}{2}}\hat\partial_2$.
Furthermore, $\psi_T$ is just defined such that $\psi_T(x) dx=\psi(\hat x) d\hat x$; 
likewise, $d$ is defined such that $d(x,0)=T^{1/3}d(\hat x,0)$.  

\medskip

{\it Step} \arabic{L2}.\label{positive}\refstepcounter{L2}
{\it H\"older norms of positive exponent}. For $\alpha\in(0,1)\cup(1,\frac{3}{2})$ and any periodic distribution $f$ we claim
\begin{align*}
[f]_\alpha\lesssim\sup_{T>0}(T^{1/3})^{-\alpha}\|T{\mathcal A}f_T\|,
\end{align*}
with the (implicit) understanding that $f$ is a continuous function (even continuously differentiable
in $x_1$ in case of $\alpha\in(1,\frac{3}{2})$) if the RHS is finite.
Here comes the argument: By homogeneity we may assume $\sup_{T>0}(T^{1/3})^{-\alpha}\|T{\mathcal A}f_T\|$ $\le 1$;
by the semi-group property (\ref{f24}) in form of $\partial_1^j\partial_2^lf_T$ 
$=\partial_1^j\partial_2^l\psi_\frac{T}{2}*f_\frac{T}{2}$ for all integers $j,l\ge 0$ and Step \ref{scaling} 
we may upgrade our assumption to
\begin{align}\label{eq12}
\|\partial_1^j\partial_2^l{\mathcal A}f_T\|\lesssim \frac{1}{T}(T^{1/3})^{-j-\frac{3}{2}l+\alpha}
\quad\mbox{for all}\;T\in(0,\infty).
\end{align}
Reasoning via $\partial_1^j\partial_2^lf_T$
$=\partial_1^j\partial_2^l\psi_{T-1}*f_1$ and using the finiteness of $\|f_1\|$, we have
\begin{align}\label{eq13}
\lim_{T\uparrow\infty}\|\partial_1^j\partial_2^lf_T\|=0\quad\mbox{provided}\;j+l>0,
\end{align}
which we need as a purely qualitative ingredient.
From the form $\psi_T(k)=\exp(-T(|k_1|^3+k_2^2))$, cf. (\ref{eq14bis}), 
we learn that $(0,\infty)\times\mathbb{R}^2\ni (T,x)\mapsto\psi_T(x)$
is a smooth solution of $(\partial_T+{\mathcal A})\psi_T=0$. 
Since by Step \ref{scaling}, $x\mapsto \psi_T(x)$ and all its derivatives are integrable,
also all its derivatives in $T$ are integrable in $x$. Hence for our periodic distribution $f$,
also $(0,\infty)\times\mathbb{R}^2\ni (T,x)\mapsto f_T(x)$ is a smooth solution of $(\partial_T+{\mathcal A})f_T=0$,
so that we have the representation
$\partial_1^j\partial_2^l(f_t-f_T)$ $=\int_t^T\partial_1^j\partial_2^l{\mathcal A}f_sds$
and thus by (\ref{eq12}) the estimate
\begin{align*}
\|\partial_1^j\partial_2^l(f_t-f_T)\|\lesssim\int_t^T(s^{1/3})^{-j-\frac{3}{2}l+\alpha}\frac{ds}{s}
\quad\mbox{for all}\;0<t<T<\infty.
\end{align*}
We use this estimate in two ways: On the one hand, 
\begin{align}\label{eq14}
\|\partial_1^j\partial_2^l(f_\tau-f_T)\|\lesssim (T^{1/3})^{-j-\frac{3}{2}l+\alpha}
\quad\mbox{provided}\;\alpha>j+\frac{3}{2}l
\end{align}
for all $T>\tau$. {In particular, for $j=l=0$, by passing to the limit $\tau<T\to 0$, we deduce that $f$ is a {continuous} function and the inequality \eqref{eq14} holds if one 
replaces $f_\tau$ by $f$.}
On the other hand, appealing to (\ref{eq13}),
\begin{align}\label{eq15}
\|\partial_1^j\partial_2^l f_T\|\lesssim (T^{1/3})^{-j-\frac{3}{2}l+\alpha}
\quad\mbox{provided}\;\alpha<j+\frac{3}{2}l.
\end{align}

\medskip

Equipped with (\ref{eq14}) \& (\ref{eq15}), we are in the position to conclude.
We first deal with the case of $\alpha\in(0,1)$; let two points $x\not=y$ be given. From
$f_T(y)-f_T(x)$ $=\int_0^1(y-x)_1\partial_1f_T(sy+(1-s)x)ds$ $+\int_0^1(y-x)_2\partial_2f_T(sy+(1-s)x)ds$
we obtain by definition of the metric $d$ that $|f_T(y)-f_T(x)|$ 
$\le\|\partial_1f_T\|d(y,x)$ $+\|\partial_2f_T\|d(y,x)^\frac{3}{2}$ and thus by the triangle inequality
\begin{align}\label{eq16}
|f(y)-f(x)|\le 2\|f-f_T\|+\|\partial_1f_T\|d(y,x)+\|\partial_2f_T\|d(y,x)^\frac{3}{2} \quad {\textrm{for } x,y\in (0,1]^2.}
\end{align}
Inserting (\ref{eq14}) for $(j,l)=(0,0)$ (where we need $\alpha>0$) and (\ref{eq15}) for $(j,l)=(1,0),(0,1)$ 
(where we need $\alpha<1$) we obtain
\begin{align*}
|f(y)-f(x)|\lesssim(T^{1/3})^{\alpha}+
(T^{1/3})^{-1+\alpha}d(y,x)+(T^{1/3})^{-\frac{3}{2}+\alpha}d(y,x)^\frac{3}{2}  \quad {\textrm{for } x,y\in (0,1]^2.}
\end{align*}
Optimizing in $T$ through the choice of $T^{1/3}=d(y,x)$ we obtain $|f(y)-f(x)|\lesssim d(y,x)^{\alpha}$
and thus $f\in C^\alpha$. 

\medskip 

The argument for $\alpha\in(1,\frac{3}{2})$ is similar but slightly more involved; we just point out the
changes w.r.t. the previous case. First, passing to the limit $\tau<T\to 0$ in \eqref{eq14} for $j=1$, $l=0$, one deduces that $\partial_1 f$ is a continuous function. Then, we use the identity
\begin{align*}
\lefteqn{f_T(y)-f_T(x)-(y-x)_1\partial_1f_T(x)}\nonumber\\
&=\int_0^1(y_2-x_2)\partial_2f_T(y_1,sy_2+(1-s)x_2)ds\\
&+\int_0^1(y_1-x_1)^2\partial_1^2f_T(sy_1+(1-s)x_1,x_2)\, (1-s) ds,
\end{align*}
yielding the inequality
\begin{align*}
\lefteqn{|f(y)-f(x)-(y-x)_1\partial_1 f(x)|}\nonumber\\
&\le 2\|f-f_T\|+\|\partial_1(f-f_T)\|d(y,x)+\|\partial_2f_T\|d(y,x)^\frac{3}{2}+
\frac{1}{2}\|\partial_1^2f_T\|d(y,x)^2.
\end{align*}
One then appeals to (\ref{eq14}) for $(j,l)=(0,0),(1,0)$ for the first two RHS terms (where one needs $\alpha>1$) and
to (\ref{eq15}) for $(j,l)=(0,1),(2,0)$ for the last two terms (where one needs $\alpha<\frac{3}{2}$). 

\medskip

{\it Step} \arabic{L2}.\label{negative}\refstepcounter{L2}
{\it {H\"older norms of negative exponent}}. For $\beta\in(-\frac{3}{2},0)$ 
and any periodic distribution $f$ of spatial vanishing average we claim
\begin{align*}
\sup_{T>0}(T^{1/3})^{-\beta}\|f_T\|\lesssim[f]_\beta.
\end{align*}
{The same proof also concludes point ii) of Lemma \ref{Lequiv} when the convolution is carried by a general (symmetric) mask $\phi$ instead of the heat kernel $\psi$. (A different argument is given at footnote \ref{ft} below). We will also prove that the inequality remains true for periodic distributions $f$ of {\it arbitrary spatial average} provided that $\sup$ is restricted over $T\in (0,1]$
(respectively, over $\ell\in (0,1]$ in \eqref{supl_numa}).} 

{Assume also for the moment that $f$ is of vanishing average on $(0,1]^2$.} 
By Definition \ref{definitionCbeta} of $[\cdot]_\beta$, the desired inequality is a consequence of the following: For any $\alpha\in(0, \frac{3}{2})$,
any integers $j,l\ge 0$ and any periodic function $u$ we have 
\begin{align}\label{eq19}
\|\partial_1^j\partial_2^lu_T\|\lesssim (T^{1/3})^{-j-\frac{3}{2}l+\alpha}[u]_\alpha
\quad\mbox{provided}\;\alpha\leq j+\frac{3}{2}l.
\end{align}
{In fact, in case of $\beta\in(-1, 0)$, 
we need (\ref{eq19}) for $(j,l,\alpha)=(1,0,\beta+1),(0,1,\beta+\frac{3}{2})$;
in case of $\beta\in(-\frac{3}{2},-1]$, we need (\ref{eq19}) for $(j,l,\alpha)=(2,0,\beta+2),(0,1,\beta+\frac{3}{2})$.
Statement (\ref{eq19}) is an immediate consequence of Step \ref{scaling} enriched by
the obvious cancellations (which follow from integrations by parts) 
\begin{align*}
\int_{\mathbb{R}^2}\partial_1^j\partial_2^l\psi_T=0&\quad\mbox{provided}\;(j,l)\not=(0,0),\\
\int_{\mathbb{R}^2}x_1\partial_1^j\partial_2^l\psi_T=0&\quad\mbox{provided}\;(j,l)\not=(0,0),(1,0).
\end{align*}
Indeed, these allow to write in the first and second case, respectively,
\begin{align*}
\partial_1^j\partial_2^lu_T(x)&=\int_{\mathbb{R}^2}\partial_1^j\partial_2^l\psi_T(x-y)(u(y)-u(x))dy,\\
\partial_1^j\partial_2^lu_T(x)&=\int_{\mathbb{R}^2}\partial_1^j\partial_2^l\psi_T(x-y)
(u(y)-u(x)-(y-x)_1\partial_1u(x))dy,
\end{align*}
which we use for  $\alpha\in (0,1]$ and $\alpha\in (1, \frac32)$, respectively.
}

Finally, the case of $f$ of nonvanishing average comes by decomposing $$f=\int_{[0,1)^2}f\, dx+\left(f-\int_{[0,1)^2}f\, dx\right)$$ so that the desired inequality in Step \ref{negative} with $\sup$ restricted to $T\in (0,1]$ follows by using that $\beta\le 0$ and $T\le 1$.

\medskip

{\it Step} \arabic{L2}.\label{schauder}\refstepcounter{L2}
{\it Proof of Lemma \ref{L2}}. For a periodic function $u$ and a periodic distribution $f$ related by
${\mathcal L}u=Pf$ we claim the following Schauder estimate: 
$[u]_\alpha\lesssim[f]_{\alpha-2}$ for $\alpha\in(\frac{1}{2},1)\cup(1,\frac{3}{2})$.
Here comes the argument: Without loss of generality (w.l.o.g.), we may  assume that $f$ has vanishing spatial average, so that
by Step \ref{negative} we have 
\begin{align*}
\sup_{T>0}(T^{1/3})^{2-\alpha}\|f_T\|\lesssim [f]_{\alpha-2};
\end{align*}
Using the semi-group property in form of $|\partial_1|f_T$ $=(|\partial_1|\psi_\frac{T}{2})*f_\frac{T}{2}$
and appealing to Step \ref{scaling}, we upgrade this to 
\begin{align*}
\sup_{T>0}(T^{1/3})^{3-\alpha}\||\partial_1|f_T\|\lesssim [f]_{\alpha-2}.
\end{align*}
We may rewrite the relation between $u$ and $f$ as
${\mathcal A}u=|\partial_1|f$ (in a distributional sense) so that the above takes the form of
\begin{align*}
\sup_{T>0}(T^{1/3})^{-\alpha}\|T{\mathcal A}u_T\|\lesssim [f]_{\alpha-2}.
\end{align*}
Therefore the claim follows from Step \ref{positive} {where it is essential to have $\alpha\neq 1$. }
{The same argument also leads to $[Ru]_\alpha\lesssim [f]_{\alpha-2}$ by redoing the above estimates for $\partial_1 f$ and ${\mathcal A} Ru=\partial_1f$. }

\medskip

{\it Step} \arabic{L2}.\label{equiv}\refstepcounter{L2}
{\it Second half of equivalence for negative exponents}.
For $\beta\in(-\frac{3}{2},-1)\cup(-1,-\frac{1}{2})\cup(-\frac{1}{2},0)$
and any periodic distribution $f$ of vanishing average we claim
\begin{align*}
[f]_\beta\lesssim\sup_{T>0}(T^{1/3})^{-\beta}\|f_T\|.
\end{align*}
With help of Fourier series, we see that there exists a periodic distribution $u$ of vanishing average
such that ${\mathcal A}u=f$ distributionally. {W.l.o.g., we assume that $f\neq 0$.} By homogeneity we may assume 
$\sup_{T>0}(T^{1/3})^{-\beta}\|f_T\|$ $=1$, so that we have
\begin{align*}
\sup_{T>0}(T^{1/3})^{-\beta}\|{\mathcal A}u_T\|\le 1,
\end{align*}
which by the semi-group property and Step \ref{scaling} implies for all $j, \ell\geq 0$:
\begin{align}
\sup_{T>0}(T^{1/3})^{j+\frac{3}{2}l-\beta-3}\|T{\mathcal A}\partial_1^j\partial_2^lu_T\|&\lesssim 1,
\label{eq23}\\
\sup_{T>0}(T^{1/3})^{j+\frac{3}{2}l+1-\beta-3}\|T{\mathcal A}\partial_1^j\partial_2^l|\partial_1|u_T\|
&\lesssim 1.\label{eq24}
\end{align}
As $f={\cal A}u=\partial_1(\partial_1|\partial_1|u)+\partial_2(-\partial_2u)$, by definition of $[\cdot]_\beta$, for having $[f]_\beta\lesssim 1$ it suffices to show
\begin{align}
[\partial_1|\partial_1|u]_{\beta+1}&\lesssim 1\quad\mbox{for}\;\beta>-1,\label{eq20}\\
[|\partial_1|u]_{\beta+2}&\lesssim 1\quad\mbox{for}\;\beta<-1,\label{eq21}\\
[\partial_2u]_{\beta+\frac{3}{2}}&\lesssim 1.\label{eq22}
\end{align}
Estimate (\ref{eq22}) follows from (\ref{eq23}) with $(j,l)=(0,1)$ by Step \ref{positive} with
$\alpha=
\beta+\frac{3}{2}$; 
estimates (\ref{eq20}) and (\ref{eq21}) follow from (\ref{eq24}) with
$(j,l)=(1,0)$ and $(j,l)=(0,0)$, respectively. 

\medskip

{\it Step} \arabic{L2}.\label{25}\refstepcounter{L2}
{\it Proof of the norm equivalence (\ref{eq25})}. 
The difference with Steps \ref{negative} and \ref{equiv} is twofold: On the one side, the range of $T$ is restricted to $(0,1]$; on the other side, the norm equivalence is claimed for periodic $f$'s without vanishing average. To deal with the latter, we decompose $f=\int_{[0,1)^2}f\, dx+\left(f-\int_{[0,1)^2}f\, dx\right)$ and note that obviously, $$\left[\int_{[0,1)^2}f\, dx\right]_\beta,\,  \sup_{T\geq 1} (T^{1/3})^{-\beta} \left|\int_{[0,1)^2}f_T\, dx\right|\leq \left|\int_{[0,1)^2}f\, dx\right|\leq [f]_\beta, \, \sup_{T\geq 1} (T^{1/3})^{-\beta} \|f_T\|.$$ 
We now argue that for a periodic distribution $f$ (of period 1) of vanishing average we have
\begin{align*}
\|f_T\|\lesssim\exp(-T)\|f_1\|\quad\mbox{for all}\;T\ge 1.
\end{align*}
By the semi-group property it suffices to show that for any periodic function $f$ of vanishing average
\begin{align*}
\|f_T\|\lesssim\exp(-T)\|f\|\quad\mbox{for all}\;T\ge 0.
\end{align*}
Since by Step \ref{scaling}, $\|f_T\|\lesssim\|f\|$ for any $T\ge 0$, it is enough to focus on $T\ge 1$.
Using the explicit form $\psi_T(k)=\exp(-T(|k_1|^3+k_2^2))$ of the convolution kernel, cf. (\ref{eq14bis}),
we obtain because of the vanishing spatial average of $f$ in form of
$f(k=0)=0$ that $|f_T(k)|$ $=\exp(-T(|k_1|^3+k_2^2))|f(k)|$
$\le\exp(1-T)\exp(-(|k_1|^3+k_2^2))|f(k)|$, where in the second step we used $T\ge 1$ and $|k_1|^3+k_2^2\geq 1$ for every $k\in (2\pi \Z)^2\setminus \{(0,0)\}$.
Hence we obtain by the Fourier series representation of $f_T$, Cauchy-Schwarz in frequency space,
and Plancherel that
\begin{align*}
\|f_T\|&\le\sum_{k}|f_T(k)|\le\exp(1-T)\sum_{k}\exp(-(|k_1|^3+k_2^2))|f(k)|\\
&\le\exp(1-T)\big(\sum_{k}\exp(-2(|k_1|^3+k_2^2))\sum_{k}|f(k)|^2\big)^\frac{1}{2}\\
&\lesssim \exp(-T)\big(\int_{[0,1)^2}f^2\big)^\frac{1}{2}
\le\exp(-T)\|f\|.
\end{align*}
\end{proof}

\medskip

\begin{rem}
The arguments presented above yield also the following equivalences:

\begin{enumerate}
\item[1.] If $\alpha\in (0, \frac 32){ \setminus \{1\}}$ and $f$ is a periodic function in ${[0,1)^2}$, then 
$$[f]_\alpha\quad  \sim \quad  \sup_{T>0} (T^{1/3})^{-\alpha} \|T\AAA f_T\|;$$

\item[2.] If $\beta\in (-\frac 32,0){ \setminus \{-1, -\frac12\}}$ and $f$ is a periodic distribution in ${[0,1)^2}$ with vanishing average, then 
$$[f]_\beta\quad  \sim \quad  \sup_{{ T>0}} (T^{1/3})^{-\beta} \|T\AAA f_T\|;$$
the equivalence still holds if the above $\sup$ is restricted to $T\in (0,1]$.
\end{enumerate}
\end{rem}

\bigskip

\begin{rem}
\label{rem:hoelder_inclu}
We have the following inclusion of our H\"older spaces of periodic functions: 
$$\textrm{$C^\alpha\subset C^\beta$ for every $-\frac32<\beta<\alpha<\frac32$, $\alpha, \beta\neq 0$}.$$ Indeed, if $\beta>0$, we know by Lemma \ref{lem:linftyhoelder} in Appendix that $[\cdot]_\beta\lesssim [\cdot]_\alpha$. If $\alpha<0$ and {$\alpha, \beta\notin \{-1, -\frac12\}$}, then within the characterization \eqref{eq25} of a periodic distribution $f$ one has $T^{-\beta/3}\|f_T\|\leq T^{-\alpha/3}\|f_T\|$ for all $T\in (0,1]$ so that $[f]_\beta\lesssim [f]_\alpha$.  If $\beta<0< \alpha$, $\beta\neq -1, -\frac12$ and $f\in C^\alpha$, then we have by Lemma \ref{lem:linftyhoelder} and 
characterization \eqref{eq25}: 
$[f]_\alpha\geq \|f-\int_{[0,1)^2}f\, dx\|\geq \sup_{T\in (0,1]} \|f_T-\int_{[0,1)^2}f\, dx\|\geq \sup_{T\in (0,1]}T^{-\beta/3}\|f_T-\int_{[0,1)^2}f\, dx\| \gtrsim [f-\int_{[0,1)^2}f\, dx]_\beta$; in conclusion, 
$[f]_\alpha+\|f\|\gtrsim [f]_\beta$. It remains to treat the critical cases $\alpha, \beta\in \{-1, -\frac12\}$. For that, we only treat here the case $\beta=-\frac12$ and $\alpha\in (\beta, 0)$; by Definition \ref{definitionCbeta}, we consider an arbitrary decomposition  $f=\int_{[0,1)^2}f\, dx+\partial_1 g+\partial_2 h$ with $g\in C^{\alpha+1}$, 
$h\in C^{\alpha+3/2}$ of vanishing average. By Lemma \ref{lem:linftyhoelder}, we know that
$[g]_{\alpha+1}\gtrsim [g]_{\beta+1}$ as well as  $[h]_{\alpha+3/2}\gtrsim [h]_{\beta+3/2}$. Passing to infimum over all these decompositions, we deduce $[f]_\alpha\gtrsim [f]_\beta$. 
\end{rem}


\subsection{Product of $C^\alpha \cdot C^\beta$ if $\alpha+\beta>0$. Proof of Lemma \ref{L3} }

\medskip
\medskip

\begin{proof}[Proof of Lemma \ref{L3}]

Let $u\in C^\alpha$ and $f\in C^\beta$ with $\alpha>0> \beta$, $\alpha+\beta>0$, $\beta\neq -1, -\frac12$.

In the first part, we will prove \eqref{w23} for our semigroup \eqref{f24}, i.e., there exists a distribution denoted by $u f\in C^\beta$ with 
\be
\label{estim_norm}
[uf]_\beta\lesssim (\|u\|+[u]_\alpha) [f]_\beta
\ee
such that for every $T\in (0,1]$ 
\begin{align}\label{w25}
(T^{1/3})^{\alpha+\beta}[u]_\alpha[f]_\beta \gtrsim
\begin{cases} \|\lceil u, (\cdot)_T\rceil f\| &\quad \textrm{ if } \alpha \in (0,1],\\
\| \big(  \lceil u, (\cdot)_T\rceil-\partial_1 u \lceil x_1, (\cdot)_T\rceil \big) f\| &\quad \textrm{ if } \alpha \in (1, \frac32).
\end{cases}
\end{align}
In the second part, we will show how to extend \eqref{w25} to a general (symmetric) Schwartz mask $\phi$ in order to obtain \eqref{w23}.

In order to prove \eqref{w25}, without loss of generality (w.l.o.g.), we may assume that {\bf $f$ is of vanishing average}, i.e., $\int_{[0,1)^2}f\, dx=0$ {\bf and that
$[f]_\beta=1$}; indeed, first, if $f=0$ the conclusion is obvious, so by homogeneity, we can assume $[f]_\beta=1$. Second, we can replace $f$ by $f-\int_{[0,1)^2}f\, dx$ and use Step 1 to check that $\|\lceil u, (\cdot)_t\rceil 1\| \lesssim
[u]_\alpha t^{\alpha/3}$
provided that $\alpha\in (0,1]$ together with {$t^{\alpha/3}\leq [f]_\beta (t^{1/3})^{\alpha+\beta}$ for $t\in (0,1]$ as $\beta\leq 0$ and $[f]_\beta=1$} (the case $\alpha\in (1, \frac32)$ is treated by the same argument). 

\newcounter{L3} 
\refstepcounter{L3} 

\nd {\it Step \arabic{L3}.\label{commut1}\refstepcounter{L3}}
For $\alpha\in(0, \frac{3}{2})$, and any two periodic functions $u$ and $f$ we have that
\begin{align}\label{eq18}
\left.\begin{array}{cl}
\mbox{for}\;\alpha\leq 1:&\|\lceil u,(\cdot)_T\rceil f\|\\
\mbox{for}\;\alpha>1:&\|(\lceil u,(\cdot)_T \rceil -\partial_1u \lceil x_1,(\cdot)_T\rceil)f\|
\end{array}\right\}\lesssim(T^{1/3})^\alpha[u]_\alpha\|f\|, \textrm{ for all } T>0.
\end{align}
Indeed, this follows from (\ref{e03}) via the representations
\begin{align}
\label{eq100}
-( \lceil u,(\cdot)_T \rceil f)(x)&=\int_{\mathbb{R}^2}\psi_T(x-y)(u(y)-u(x))f(y)dy,\\
\label{ap11}
-(\lceil u,(\cdot)_T\rceil-\partial_1u \lceil x_1,(\cdot)_T\rceil )f(x)&
=\int_{\mathbb{R}^2}\psi_T(x-y)(u(y)-u(x)-(y-x)_1\partial_1u(x))f(y)dy
\end{align}
and the definition of $[\cdot]_\alpha$.

\nd {\it Step  \arabic{L3}.\label{commut2}\refstepcounter{L3} }
For $\beta\in(-\frac{3}{2},0)\setminus \{-1, -\frac12\}$ and any periodic distribution $f$ of spatial vanishing average we have
\begin{align}\label{eq27}
\| \lceil x_1,(\cdot)_T\rceil f\|\lesssim [f]_\beta(T^{1/3})^{1+\beta}\quad\mbox{for all}\;T>0.
\end{align}

Indeed, %
\begin{align}\label{eq30}
( \lceil x_1,(\cdot)_T \rceil f)(x)=T^{1/3}\int_{\mathbb{R}^2}\tilde\psi_T(x-y)f(y)dy=T^{1/3}\tilde\psi_T*f(x),
\end{align}
where $\tilde\psi(x)=x_1\psi(x)$ and $\tilde\psi_T$ is related to $\tilde\psi$ analogously to (\ref{f25}). 
{Note that $|\tilde \psi(x)|\leq |\psi(x)| d(x,0)$ so that using (\ref{e03}) for $(j,l,\alpha)=(0,0,1)$ and $T=1$ we deduce that $\tilde \psi\in L^1(\R^2)$. Moreover, by the scaling of $\tilde\psi_T$, it follows that $\tilde\psi_T$ and $\tilde\psi$ have the same $L^1$-norm.}
We first argue that
\begin{align}\label{eq29}
\tilde \psi_{2T}={ 2^{2/3}}\tilde\psi_T*\psi_T.
\end{align}
For that, one writes, using that $*$ is Abelian,
\begin{align*}
2\tilde \psi_T*\psi_T(x)&=\int_{\R^2} \frac{x_1-y_1}{T^{1/3}} \psi_T(x-y)\psi_T(y)\, dy+\int_{\R^2} \frac{y_1}{T^{1/3}} \psi_T(y)\psi_T(x-y)\, dy\\
&=\frac{x_1}{T^{1/3}}\psi_{2T}(x)=2^{1/3}\tilde \psi_{2T}(x).
\end{align*}
In view of (\ref{eq29}) we may rewrite (\ref{eq30}) as
\begin{align*}
\left|( \lceil x_1,(\cdot)_{2T}\rceil f)(x)\right|=\left|{ 2^{2/3}} (2T)^{1/3}\int_{\mathbb{R}^2}\tilde\psi_T(x-y)f_T(y)dy\right|\leq { 2 T^{1/3} \|f_T\| \int_{\R^2} |\tilde \psi|\, dx}.
\end{align*}

Hence, 
\begin{align*}
\| \lceil x_1,(\cdot)_{2T} \rceil f\|\lesssim T^{1/3}\|f_T\|.
\end{align*}
Now (\ref{eq27}) follows from Step \ref{negative} in the proof of Lemma \ref{Lequiv}. 

\medskip

\nd {\it Step \arabic{L3}.\label{st1}\refstepcounter{L3} } \textit{For any $\tau,T>0$, we have}
\begin{align*}
[u]_\alpha [f]_\beta (T^{1/3})^{\alpha} (\tau^{1/3})^{\beta} \gtrsim
\begin{cases} \|\lceil u, (\cdot)_T\rceil f_\tau\| &\quad \textrm{ if } \alpha \in (0,1],\\
\| \big(  \lceil u, (\cdot)_T\rceil-\partial_1 u \lceil x_1, (\cdot)_T\rceil \big) f_\tau\| &\quad \textrm{ if } \alpha \in (1, \frac32).
\end{cases}
\end{align*}
{\it In particular, if $\tau=T>0$, 
then the above RHS is bounded by} $[u]_\alpha [f]_\beta (T^{1/3})^{\alpha+\beta}$. 

Indeed, this is a direct consequence of Step 1 and Lemma \ref{Lequiv}.

\nd {\it Step \arabic{L3}.\label{st2}\refstepcounter{L3} } \textit{For every $0<t < T$, we have }
\[  [u]_\alpha [f]_\beta  (T^{1/3})^{\alpha+\beta}\gtrsim  \begin{cases} \|\lceil u, (\cdot)_{T-t}\rceil f_t\| &\, \textrm{ if } \alpha \in (0,1],\\
\| \lceil u, (\cdot)_{T-t}\rceil f_t-\partial_1 u \lceil x_1, (\cdot)_{T}\rceil f+\big(\partial_1 u \lceil x_1, (\cdot)_{t}\rceil f\big)_{T-t}\| &\, \textrm{ if } \alpha \in (1, \frac32).
\end{cases}\]
To prove that, we distinguish two cases: 

{\it The dyadic case}. We start with the case of $t$ and $T$ being dyadically related (i.e., $t=T/2^n$).
By the semigroup property \eqref{f24}, we have that $(\lceil u, (\cdot)_{\tau}\rceil f_\tau)_{T-2\tau}=\big(u f_{2\tau}-(uf_\tau)_\tau\big)_{T-2\tau}=
(u f_{2\tau})_{T-2\tau}-(uf_\tau)_{T-\tau}$ which leads to a telescopic sum:
\be
\label{1dyad}
\lceil u, (\cdot)_{T-t}\rceil f_t=uf_T-(uf_t)_{T-t}=\sum_{\tau=T/2^k, \, k=1, \dots, n } (\lceil u, (\cdot)_{\tau}\rceil f_\tau)_{T-2\tau}.
\ee
$\bullet$ If $\alpha\in (0,1]$, then by Step 3, 
\begin{align*} 
\lVert \lceil u, (\cdot)_{T-t}\rceil f_t \rVert & \leq \sum_{\tau=T/2^k, \, k=1, \dots, n } \| \lceil u, (\cdot)_{\tau}\rceil f_\tau\|\\
& \lesssim  [u]_\alpha [f]_\beta \sum_{\tau=T/2^k, \, k=1, \dots, n } (\tau^{1/3})^{\alpha+\beta}\\
& \lesssim  [u]_\alpha [f]_\beta  (T^{1/3})^{\alpha+\beta} \sum_{k\geq 1} (\tfrac{1}{2^{(\alpha+\beta)/3}})^k \stackrel{\alpha+\beta>0}{\lesssim}  [u]_\alpha [f]_\beta (T^{1/3})^{\alpha+\beta}. 
\end{align*}

$\bullet$ If $\alpha\in (1, \frac32)$, then simple algebra yields
\be
\label{tstar}
 -\partial_1 u \lceil x_1, (\cdot)_{2\tau}\rceil f + \big(\partial_1 u \lceil x_1, (\cdot)_{\tau}\rceil f\big)_\tau
=-\partial_1 u \lceil x_1, (\cdot)_{\tau}\rceil f_{\tau}-\lceil \partial_1 u, (\cdot)_{\tau}\rceil 
\lceil x_1, (\cdot)_{\tau}\rceil f
\ee
which when convoluted with $\psi_{T-2\tau}$ leads to another telescopic sum:
\begin{align*} 
&-\partial_1 u \lceil x_1, (\cdot)_{T}\rceil f + \big(\partial_1 u \lceil x_1, (\cdot)_{t}\rceil f\big)_{T-t}\\
&\quad \quad=-\sum_{\tau=T/2^k, \, k=1, \dots, n } \big(\partial_1 u \lceil x_1, (\cdot)_{\tau}\rceil f_{\tau}+\lceil \partial_1 u, (\cdot)_{\tau}\rceil \lceil x_1, (\cdot)_{\tau}\rceil f \big)_{T-2\tau}.
\end{align*}
Adding this to \eqref{1dyad}, we obtain:
\begin{align*}
&\lceil u, (\cdot)_{T-t}\rceil f_t-\partial_1 u \lceil x_1, (\cdot)_{T}\rceil f+
\big(\partial_1 u \lceil x_1, (\cdot)_{t}\rceil f\big)_{T-t}\\
&=\sum_{\tau=T/2^k, \, k=1, \dots, n } \big(\lceil u, (\cdot)_{\tau}\rceil f_\tau-\partial_1 u \lceil x_1, (\cdot)_{\tau}\rceil f_{\tau}-\lceil \partial_1 u, (\cdot)_{\tau}\rceil \lceil x_1, (\cdot)_{\tau}\rceil f \big)_{T-2\tau}.
\end{align*}
The first contribution to the summand is estimated by Step 3: 
$$\|\big( \lceil u, (\cdot)_{\tau}\rceil -\partial_1 u \lceil x_1, (\cdot)_{\tau}\rceil \big)f_{\tau}\|\lesssim [u]_\alpha [f]_\beta 
(\tau^{1/3})^{\alpha+\beta};$$
the second contribution is estimated by Steps 1-2 and Lemma \ref{lem:linftyhoelder}:
\be
\label{claimul}
\| \lceil \partial_1 u, (\cdot)_{\tau}\rceil 
\lceil x_1, (\cdot)_{t}\rceil f\|\lesssim [\partial_1 u]_{\alpha-1} \tau^{\frac{\alpha-1}3}  \,
\| \lceil x_1, (\cdot)_{t}\rceil f\| \lesssim 
[f]_\beta [u]_{\alpha} \tau^{\frac{\alpha-1}3} t^{\frac{\beta+1}3},
\ee
for every $t, \tau>0$.
The desired estimate follows as in the case $\alpha\in (0,1]$.

\medskip
{\it The nondyadic case}.
In the general case of $t$ not dyadically related to $T$, we choose $\ttt \in [T/2, T)$ that is dyadically related to $t$ (in particular, $\ttt\geq t$). 

$\bullet$ If $\alpha\in (0,1]$, then we have
\be
\label{dstar}
 \lceil u, (\cdot)_{T-t}\rceil f_t=\big( \lceil u, (\cdot)_{\ttt-t}\rceil f_t\big)_{T-\ttt}+ \lceil u, (\cdot)_{T-\ttt}\rceil f_\ttt
 \ee
  so that by the dyadic case and Step 3, we conclude:
$$\| \lceil u, (\cdot)_{T-t}\rceil f_t\|\lesssim \| \lceil u, (\cdot)_{\ttt-t}\rceil f_t\|+\| \lceil u, (\cdot)_{T-\ttt}\rceil f_\ttt\|\lesssim  [u]_\alpha [f]_\beta (T^{1/3})^{\alpha+\beta}.$$

$\bullet$ If $\alpha\in (1, \frac32)$, starting from \eqref{dstar}, we have
\begin{align}
\nonumber & \lceil u, (\cdot)_{T-t}\rceil f_t-\partial_1u \lceil x_1, (\cdot)_{T}\rceil f+\big(\partial_1u \lceil x_1, (\cdot)_{t}\rceil f \big)_{T-t}\\
\label{no1}&=\bigg( \lceil u, (\cdot)_{\ttt-t}\rceil f_t-\partial_1u \lceil x_1, (\cdot)_{\ttt}\rceil f+ \big(\partial_1u \lceil x_1, (\cdot)_{t}\rceil f \big)_{\ttt-t}   \bigg)_{T-\ttt}\\
\label{no2} &+ \lceil u, (\cdot)_{T-\ttt}\rceil f_\ttt-\partial_1u \lceil x_1, (\cdot)_{T}\rceil f+\big(\partial_1u \lceil x_1, (\cdot)_{\ttt}\rceil f \big)_{T-\ttt} .
\end{align}
The dyadic case (applied to $t$ and $\ttt$) yields the estimate of the first term \eqref{no1} by $\lesssim  [u]_\alpha [f]_\beta (T^{1/3})^{\alpha+\beta}$, while the second \eqref{no2} is estimated using an identity similar to the computation \eqref{tstar}:
\begin{align*}
&\| \lceil u, (\cdot)_{T-\ttt}\rceil f_\ttt-\partial_1u \lceil x_1, (\cdot)_{T}\rceil f+\big(\partial_1u \lceil x_1, (\cdot)_{\ttt}\rceil f \big)_{T-\ttt}\|\\
&\leq \|\lceil u, (\cdot)_{T-\ttt}\rceil f_\ttt-\partial_1 u  \lceil x_1, (\cdot)_{T-\ttt}\rceil f_\ttt\|+\|\lceil \partial_1 u, (\cdot)_{T-\ttt}\rceil \lceil x_1, (\cdot)_{\ttt}\rceil f\|\\
&\lesssim [u]_\alpha [f]_\beta \big( (T-\ttt)^{\alpha/3} \ttt^{\beta/3}+(T-\ttt)^{(\alpha-1)/3} \ttt^{(\beta+1)/3}\big)\lesssim [u]_\alpha [f]_\beta T^{(\alpha+\beta)/3}
\end{align*}
where we used Step 3 and \eqref{claimul}. 

\nd {\it Step \arabic{L3}.\label{st3}\refstepcounter{L3} }\textit{For a  subsequence, 
$$
\begin{cases}
\big\{uf_{\frac{1}{2^n}}\big\}_{n \uparrow \infty} &\quad \textrm{ if } \alpha\in (0,1],\\ 
\big\{uf_{\frac{1}{2^n}}-\partial_1 u  \lceil x_1, (\cdot)_{\frac{1}{2^n}}\rceil f\big \}_{n \uparrow \infty} &\quad \textrm{ if } \alpha\in (1, 3/2),
\end{cases}
$$
converges (in a distributional sense) to a distribution denoted 
$uf$ that belongs to $C^\beta$ such that \eqref{estim_norm} and \eqref{w25} hold.}

$\bullet$ If $\alpha\in (0,1]$, we set $t=1/2^n$. We want to prove that $[uf_t]_\beta\lesssim ([u]_\alpha+\|u\|) [f]_\beta$ which by Lemma \ref{Lequiv} {(as we assumed $\beta\neq -\frac12, -1$)}, it is equivalent to checking that
\be
\label{eq77}
\|(uf_t)_T\|\lesssim ([u]_\alpha+\|u\|) [f]_\beta T^{\beta/3}, \quad \textrm{for all }\, T\in (0,1]. 
 \ee
If $T\in (0, t]$, then by Step 3 and Lemma \ref{Lequiv}, we have 
\begin{align*}
\|(uf_t)_T\|&\leq \|(uf_t)_{T}-u f_{T+t}\|+\|uf_{T+t}\|\\
&\lesssim [u]_\alpha [f]_\beta (T^{1/3})^{\alpha}(t^{1/3})^{\beta}+\|u\|[f]_\beta (T+t)^{\beta/3} \lesssim (\|u\|+[u]_\alpha) [f]_\beta (T^{1/3})^{\beta}
\end{align*}
because of $\alpha>0>\beta$, $(T^{1/3})^{\alpha}\leq 1$ and $(t^{1/3})^{\beta}, (T+t)^{\beta/3} \leq (T^{1/3})^{\beta}$.

If $T\in (t, 1]$, then we have by Step 4 and Lemma~\ref{Lequiv},
\begin{align*}
\|(uf_t)_T\|= &\|\big((uf_t)_{T-t}\big)_{t}\|\lesssim \|(uf_t)_{T-t}\|\leq \|(uf_t)_{T-t}-u f_{T}\|+\|uf_{T}\|  \\
&\lesssim [u]_\alpha [f]_\beta (T^{1/3})^{\alpha+\beta}+\|u\|[f]_\beta T^{\beta/3} \lesssim ([u]_\alpha+\|u\|) [f]_\beta T^{\beta/3}.
\end{align*}
Therefore, \eqref{eq77} holds. By Lemma \ref{lem:compactness} in Appendix, for a subsequence $n\to \infty $ with $t=1/2^n$, we have $uf_t$ converges to a distribution that we denote $uf$ which belongs to $C^\beta$ and 
$[uf]_\beta\lesssim   ([u]_\alpha+\|u\|) [f]_\beta$, so \eqref{estim_norm} holds. Moreover, for every $T>0$ we have $\psi_{T-t}\to \psi_T$ in the sense of Schwartz functions as $t\to 0$ so that $(uf_t)_{T-t}\to (uf)_T$ uniformly. By Step 4, passing at the limit $t\to 0$, we conclude $\lVert (uf)_T-uf_T \rVert \lesssim [u]_\alpha [f]_\beta (T^\frac{1}{3})^{\alpha+\beta}$ which is \eqref{w25}.

$\bullet$ If $\alpha\in (1, 3/2)$, we want to show  
$$\|\big(uf_{t}-\partial_1 u  \lceil x_1, (\cdot)_{t}\rceil f\big)_T\| \lesssim ([u]_\alpha+\|u\|) [f]_\beta T^{\beta/3}, \quad \textrm{for all }\, T\in (0,1],
$$
for every $t=1/2^n$, $n\in \N$. 

If $T\in (0, t]$, then by Step 3, \eqref{eq27} and Lemma \ref{Lequiv} we have 
\begin{align*}
&\|\big(uf_t-\partial_1 u  \lceil x_1, (\cdot)_{t}\rceil f\big)_T\|= \|(uf_t)_T-uf_{T+t}+uf_{T+t}-\big(\partial_1 u  \lceil x_1, (\cdot)_{t}\rceil f\big)_T\|  \\
&\lesssim \|\lceil u, (\cdot)_{T}\rceil f_t-\partial_1u \lceil x_1, (\cdot)_{T}\rceil f_t\|+\|\partial_1u \lceil x_1, (\cdot)_{T}\rceil f_t\|+
\|uf_{T+t}\|+\|\partial_1 u  \lceil x_1, (\cdot)_{t}\rceil f \|\\
&\lesssim [u]_\alpha [f]_\beta (T^{1/3})^{\alpha}(t^{1/3})^{\beta} +\|\partial_1 u\|[f]_\beta T^{1/3}(T+t)^{\beta/3}+\|u\|[f]_\beta (T+t)^{\beta/3}+\|\partial_1 u\|[f]_\beta t^{(\beta+1)/3}\\
&\lesssim ([u]_\alpha+\|u\|) [f]_\beta  (T^{1/3})^{\beta}
\end{align*}
where we used $\|\partial_1 u\|\lesssim [u]_\alpha$ (by Lemma \ref{lem:linftyhoelder}) and a slightly different version of \eqref{eq27}: 
$$\|\lceil x_1, (\cdot)_{T}\rceil f_t\|\lesssim [f]_\beta T^{1/3} (T+t)^{\beta/3}.$$ In fact, for the latter estimate, we use the same strategy as in Step 2: by \eqref{eq30}-\eqref{eq29} and Lemma \ref{Lequiv}, 
we have that $\|\lceil x_1, (\cdot)_{2T}\rceil f_t\| \lesssim T^{1/3} \|f_{t+T}\| \lesssim [f]_\beta T^{1/3} (T+t)^{\beta/3}$.

If $T\in (t, 1]$, then we have by Step 4, Lemma \ref{Lequiv} and \eqref{eq27},
\begin{align*}
&\|\big(uf_t-\partial_1 u  \lceil x_1, (\cdot)_{t}\rceil f\big)_T\|\lesssim \|(uf_t)_{T-t}-\big(\partial_1 u  \lceil x_1, (\cdot)_{t}\rceil f\big)_{T-t}\|\\
&\lesssim \|\lceil u, (\cdot)_{T-t}\rceil f_t -\partial_1 u \lceil x_1, (\cdot)_{T}\rceil f+\big(\partial_1 u \lceil x_1, (\cdot)_{t}\rceil f\big)_{T-t} \|
+\|\partial_1 u \lceil x_1, (\cdot)_{T}\rceil f\|+\|uf_{T}\|  \\
&\lesssim [u]_\alpha [f]_\beta (T^{1/3})^{\alpha+\beta}+[u]_\alpha [f]_\beta T^{(\beta+1)/3}+\|u\|[f]_\beta T^{\beta/3}\\ 
&\lesssim ([u]_\alpha+\|u\|) [f]_\beta T^{\beta/3}.
\end{align*}
where we used again $\|\partial_1 u\|\lesssim [u]_\alpha$. As {we assumed $\beta\neq -\frac12, -1$}, by 
Lemma \ref{Lequiv} we have that
$[uf_t-\partial_1 u  \lceil x_1, (\cdot)_{t}\rceil f]_\beta\lesssim ([u]_\alpha+\|u\|) [f]_\beta$.

By Lemma \ref{lem:compactness} in Appendix, for a subsequence $n\to \infty $ with $t=1/2^n$, we have $uf_t-\partial_1 u  \lceil x_1, (\cdot)_{t}\rceil f\rightharpoonup uf$ for some distribution denoted by $uf$ that belongs to $C^\beta$ and 
$[uf]_\beta\lesssim   ([u]_\alpha+\|u\|) [f]_\beta$, so \eqref{estim_norm} holds. Moreover, for every $T>0$ we have $\psi_{T-t}\to \psi_T$ in the sense of Schwartz functions as $t\to 0$ so that $(uf_t)_{T-t}- \big(\partial_1 u  \lceil x_1, (\cdot)_{t}\rceil f\big)_{T-t}\to (uf)_T$ uniformly. By Step 4, passing at the limit $t\to 0$, we conclude \eqref{w25}, i.e.,  $$\lVert \underbrace{(uf)_T-uf_T}_{=\,-\lceil u, (\cdot)_{T}\rceil f} + \partial_1 u \lceil x_1, (\cdot)_{T}\rceil f \rVert \lesssim [u]_\alpha [f]_\beta (T^\frac{1}{3})^{\alpha+\beta}.$$

\nd {\it Step \arabic{L3}.\label{st5}\refstepcounter{L3} We prove (\ref{w23}) for a general (symmetric) Schwartz kernel $\phi$}. 
Indeed, this relies on the following
representation 
of $\phi$ in terms of $\psi_t$ via a family $\{\omega^t\}_{t\in (0,1]}$ of smooth functions
of sufficient decay:
\begin{align}\label{w24}
\phi=\int_0^1\omega^t*\psi_t dt\quad\mbox{and}\quad
\int_{\R^2} (1+d^2(x,0))|\omega^t(x)|\, \, dx\lesssim 1, \quad t\in (0,1],
\end{align}
where $\omega^t$ 
$:=\big(1+(1-t)\AAA\big)\phi$.
The formula in (\ref{w24}) follows via $\phi$ $=\int_0^1\frac{d}{dt}((t-1)\phi_t)dt$ and the definition 
$\phi_t=\phi*\psi_t$ from the characterization of the kernel $\psi_t$ in form of
$\partial_t\psi_t=-\AAA \psi_t$. Since $\phi$ is a Schwartz function, the only
issue with the estimate in (\ref{w24}) is the non-local term $\int(1+x_1^2)||\partial_1|^3\phi|\, \, dx_1$,
which is controlled as in Step \ref{moments} in the proof of Lemma \ref{Lequiv} by the convergent
$(\int(1+x_1^2)^{-1}\, \, dx_1)^\frac{1}{2}$ times $(\int(1+x_1^6)(|\partial_1|^3\phi)^2\, \, dx_1)^\frac{1}{2}$.
By Plancherel, the latter can be rewritten as $\sum_{l=0,3}(\int|\partial_{k_1}^l |k_1|^3\phi|^2dk_1)^\frac{1}{2}$,
which itself is dominated by $\sum_{l=0,3}\sum_{n=0,\cdots,l}(\int|k_1^{3-n}\partial_{k_1}^n\phi|^2dk_1)^\frac{1}{2}$.
Since $\phi$ is Schwartz, the integral over $x_2$ of this expression is bounded. Moreover, using the scaling \eqref{f37} (in particular, $(\psi_t)_{\ell}=\psi_{t\ell^3}$), the formula in (\ref{w24}) yields 
$$\phi_\ell=\int_0^1\omega^{t}_\ell*\psi_{t\ell^3} dt, \quad \textrm{ with } \quad
\omega^t_\ell(x_1,x_2)=\ell^{-\frac{5}{2}}\omega^t(\ell^{-1}x_1,\ell^{-\frac{3}{2}}x_2),$$
where the subscript $\ell$ in $\omega^t_\ell$ denotes the rescaling like for $\phi_\ell$.
\footnote{ \label{ft} The representation  (\ref{w24}) gives a new proof of point ii) in Lemma \ref{Lequiv} by passing from the estimate \eqref{eq25} on $\|f_T\|$ to the desired estimate \eqref{supl_numa} on $\|f_\ell\|$. Indeed, 
we convert convolution with $\phi_\ell$ into convolution with $\psi_t$ and deduce for $\ell\leq 1$:
\begin{align*}
\|f_\ell\|&=\|f*\int_0^1\omega^{t}_\ell*\psi_{t\ell^3} dt\|\leq \int_0^1  \|\omega^{t}_\ell*f_{t\ell^3}\| dt \leq \int_0^1 \|f_{t\ell^3}\| \|\omega^{t}_\ell\|_{L^1}\, dt\lesssim [f]_\beta \int_0^1 (t^{\frac13}\ell)^\beta \, dt\lesssim  [f]_\beta \ell^\beta,
\end{align*}
as $\beta>-3$. }

Coming back to the proof of (\ref{w23}), as we will use several mollifiers, in order to avoid confusion with $\lceil u, (\cdot)_{\ell \, (\textrm{or } t)}\rceil$, we introduce the following notation for the commutator-convolution:
$$\lceil u, \phi*\rceil f :=u \phi*f-\phi*(uf).$$
Then one checks (by simple algebra) the following:
\be
\label{opera}
\lceil u, \om*\psi*\rceil f=\om*\lceil u, \psi*\rceil f+\lceil u, \om*\rceil (\psi*f).
\ee

Combined with \eqref{w24},
this allows us to convert convolution with $\phi_\ell$ into convolution with $\psi_t$: 
\begin{align}
\nonumber
\lceil u, \phi_\ell*\rceil f&=u f_\ell-(uf)_\ell=\int_0^1 u\omega_\ell^t*f_{t\ell^3}-\omega_\ell^t*(uf)_{t\ell^3}\, dt\nonumber\\
\label{eq50}
&=\int_0^1 \lceil u, \omega_\ell^t*\psi_{t\ell^3}*\rceil f\, dt\stackrel{\eqref{opera}}{=}
\int_0^1\Big(\omega_\ell^t*\lceil u, \psi_{t\ell^3}*\rceil f+\lceil u,\omega_\ell^t*\rceil f_{t\ell^3}\Big)dt.
\end{align}
$\bullet$ If $\alpha\in (0,1]$, 
 by the same arguments, we estimate
\begin{align*}
\|\lceil u, \phi_\ell*\rceil f\|&
\stackrel{\eqref{eq50}}{\lesssim}\int_0^1\big(\|\lceil u, \psi_{t\ell^3}*\rceil f\|+[u]_\alpha\ell^\alpha\|f_{t\ell^3}\|\big)dt.
\end{align*}
We now appeal to (\ref{w25}) and Lemma \ref{Lequiv}  to obtain
\begin{align*}
\|\lceil u, \phi_\ell*\rceil f\|=\|u f_\ell-(uf)_\ell\|&\lesssim[u]_\alpha[f]_\beta
\int_0^1\big((t^\frac{1}{3}\ell)^{\alpha+\beta}+\ell^\alpha (t^\frac{1}{3}\ell)^\beta\big)dt, \quad \forall \ell\in (0,1].
\end{align*}
Because in particular $\beta>-3$, this implies (\ref{w23}). It remains to prove that $uf_\ell \rightharpoonup uf$ in 
${\cal D}'$ where $uf$ is the distribution defined by \eqref{estim_norm}, which in particular shows the uniqueness of the limit $uf$ independently of the symmetric mask $\phi$. Indeed, for every smooth periodic test function $\zeta$, we have that
\begin{align*}
\int_{[0,1)^2} uf_\ell \zeta\, dx&=\int_{[0,1)^2} \big(uf_\ell-(uf)_\ell\big)  \zeta\, dx+\int_{[0,1)^2}(uf)_\ell \zeta\, dx\\
&\stackrel{\eqref{w23}}{\lesssim} [u]_\alpha [f]_\beta \ell^{\alpha+\beta}\|\zeta\|_{L^1} + \langle uf, \zeta*\phi_\ell \rangle_{{\cal D}', {\cal D}}
\longrightarrow \langle uf, \zeta \rangle_{{\cal D}', {\cal D}} \quad \textrm{ as } \ell \to 0,
\end{align*}
where we used that ${\phi}$ is symmetric.

\medskip

$\bullet$ If $\alpha\in (1, 3/2)$, starting from \eqref{eq50}, we have by \eqref{w24}:
\begin{align*}
\lceil u, \phi_\ell*\rceil f-\partial_1 u \lceil x_1, \phi_\ell*\rceil f &=\int_0^1\Big(\omega_\ell^t*\lceil u, \psi_{t\ell^3}*\rceil f+\lceil u,\omega_\ell^t*\rceil f_{t\ell^3}-
\partial_1 u \lceil x_1, \omega_\ell^t* \psi_{t\ell^3}*\rceil f\Big)dt\\
&\stackrel{\eqref{opera}}{=}\int_0^1\Bigg(\omega_\ell^t*\big(\lceil u, \psi_{t\ell^3}*\rceil f-\partial_1 u \lceil x_1, \psi_{t\ell^3}*\rceil f\big)\\
&+\big(\lceil u,\omega_\ell^t*\rceil f_{t\ell^3}-\partial_1 u \lceil x_1, \omega_\ell^t*\rceil f_{t\ell^3}\big)
-\lceil\partial_1 u, \omega_\ell^t*\rceil \lceil x_1, \psi_{t\ell^3}*\rceil f\Bigg)dt.
\end{align*}
Now, by the same arguments using \eqref{ap11} and \eqref{claimul}, we estimate
\begin{align*}
&\|\big(\lceil u, \phi_\ell*\rceil -\partial_1 u \lceil x_1, \phi_\ell*\rceil \big) f\|\\
&\lesssim\int_0^1\big(\|(\lceil u, \psi_{t\ell^3}*\rceil -\partial_1 u \lceil x_1, \psi_{t\ell^3}*\rceil)f\|+
[u]_\alpha\ell^\alpha\|f_{t\ell^3}\|+[u]_\alpha\ell^{\alpha-1}\|\lceil x_1, \psi_{t\ell^3}*\rceil f\|\big)dt.
\end{align*}
We now appeal to (\ref{w25}), \eqref{eq27} and Lemma \ref{Lequiv} to obtain
\begin{align*}
\|\big(\lceil u, \phi_\ell*\rceil -\partial_1 u \lceil x_1, \phi_\ell*\rceil\big) f\|&\lesssim[u]_\alpha[f]_\beta
\int_0^1\big((t^\frac{1}{3}\ell)^{\alpha+\beta}+\ell^\alpha (t^\frac{1}{3}\ell)^\beta+\ell^{\alpha-1} (t^\frac{1}{3}\ell)^{\beta+1}\big)dt\\
&\lesssim [u]_\alpha[f]_\beta \ell^{\alpha+\beta},
\end{align*}
that is \eqref{w23}. It follows that $u f_\ell-\partial_1 u \lceil x_1, \phi_\ell*\rceil f\rightharpoonup uf$ distributionally as in the case $\alpha\in (0,1]$.
\end{proof}


\subsection{Regularity of the Hilbert transform. Proof of Lemma \ref{L4}.}

Since the Hilbert transform $R$ acts only on the $x_1$-variable, it does not map $C^{\alpha}$ into $C^\alpha$ on the two-dimensional torus $[0,1)^2$, but in a slightly
larger H\"older space (corresponding to a smaller exponent $\alpha-\eps$ for any $\eps>0$). This is the result in Lemma \ref{L4}.

\begin{proof}[Proof of  Lemma \ref{L4}]
For the reader's convenience, we start by proving the boundedness of the Hilbert transform $R$ over the single-variable H\"older space $C^\alpha_{x_1}$. Even if the result is standard, we want to highlight that the method based on the ``heat kernel'' $\psi_T$ in proving Lemma \ref{L2} can be adapted here by using a different ``heat kernel''. 
More precisely, we introduce 
$G_T(x_1)=\frac{1}{T} G(\frac{x_1}{T})$ where the semigroup kernel $G$ is given in Fourier space: $$G(k_1)=e^{-|k_1|} \, \textrm{ for } \, k_1\in \R,$$ 
so that $G_T$ is the ``heat kernel'' of the semigroup generated by $|\partial_1|$. As for the semigroup $\psi_T$ in \eqref{eq14bis}, 
while $G(x_1)$ is a smooth (since $G(k_1)$ decays exponentially), its decay is moderate (since $G(k_1)$ is only Lipschitz). 
More precisely, we claim for $\alpha\in (0, \frac32)$: 
\be
\label{eq:estG}
\int_\R |G_T|  \, dx_1\lesssim 1, \, \int_\R |\partial_1 G_T| \, dx_1
\lesssim T^{-1}, \,
\int_\R |x_1|^\alpha\, |\partial_1 G_T|  \,  dx_1\lesssim T^{\alpha-1}, \,  \int_\R \big||\partial_1| G_T\big|\, \, dx_1\lesssim T^{-1}.\ee 
By scaling, it is enough to consider $T=1$. By Cauchy-Schwarz, the above reduces to 
$$\int_\R (x_1^2+1)G^2  \,  dx_1, \, \int_\R (x_1^4+1) (\partial_1 G)^2 \, dx_1, \int_\R (x_1^2+1)\, (|\partial_1| G)^2   \, dx_1\lesssim 1,  $$
which by Plancherel is equivalent to
$$\int_\R (\partial_{k_1}G)^2+G^2  \, \, dk_1, \, \int_\R \big(\partial_{k_1}^2(k_1G)\big)^2+(k_1 G)^2 \, dk_1, \int_\R 
\big(\partial_{k_1}(|k_1|G)\big)^2+(k_1 G)^2   \, dk_1\lesssim 1,  $$
which holds since $k_1\mapsto G(k_1), |k_1|G(k_1), k_1 G(k_1), \partial_{k_1} (k_1 G(k_1)) $ are Lipschitz and the boundedness of 
the above quantities follows in combination with the exponential decay of $G$.

\nd {\it Step 1. If $\alpha\in (0, \frac32)\setminus \{1\}$, then the Hilbert transform on the one-dimensional torus $[0,1)$ satisfies
\be
\label{rectang}
[Ru]_\alpha\lesssim [u]_\alpha
\ee
for any periodic function $u=u(x_1)$ of vanishing average.}

$\bullet$ If $\alpha\in (0,1)$, let $u=u(x_1)\in C_{x_1}^\alpha$ with $\int_0^1 u\, \, dx_1=0$ and $f:=Ru$. We want to prove that $f\in C_{x_1}^\alpha$ with 
$[f]_{\alpha}\lesssim [u]_{\alpha}$. We will follow the main lines in Step \ref{positive} of the proof of Lemma \ref{Lequiv}. More precisely, let $u_T=G_T*u$ and $f_T=G_T*f $. \footnote{Do not confound with the notation \eqref{f24} that uses the semigroup $\psi_T$ defined in \eqref{f25}. We use the semigroup $G_T$ only at this step.} Note that
\be
\label{eq:loi_grupG}
\partial_T f_T=-|\partial_1| f_T=-\partial_1 u_T.\ee
As $\int_\R \partial_1 G_T\, \, dx_1=0$, we have
\begin{align*}
&\partial_1 u_T(x_1)=\int_{\R} \partial_1 G_T(y_1) (u(x_1-y_1)-u(x_1))\, dy_1\\ 
 \Rightarrow\, \quad &
\|\partial_1 u_T\|\leq [u]_\alpha \int_\R  |y_1|^\alpha |\partial_1 G_T|  \, \, dy_1
\stackrel{\eqref{eq:estG}}{\lesssim} T^{\alpha-1}[u]_\alpha
\end{align*}
as well as $\ds \|\partial^2_1 u_T\|=\| \partial_1u_{T/2}*\partial_1 G_{T/2}\|\leq \|\partial_1 u_{T/2}\| \int_{\R} |\partial_1 G_{T/2}| \, dx_1\stackrel{\eqref{eq:estG}}{\lesssim} 
T^{\alpha-2}[u]_\alpha$. 
Since $\alpha>0$, it follows for every $0<\tau<T$:
\be
\label{ab1}
\|f_T-f_\tau\|= \|\int_\tau^T \partial_s f_s\, ds\|\stackrel{\eqref{eq:loi_grupG}}{\leq} \int_\tau^T \|\partial_1 u_s\|\, ds\lesssim T^\alpha [u]_\alpha
\ee
which in particular proves by passing to the limit $\tau<T\to 0$ that $f$ is a continuous function. Moreover, by \eqref{eq:estG},
$\ds \|\partial_1 f_T\|=\|u* |\partial_1| G_T\|\leq \| u \|  \int_{\R}\big| |\partial_1| G_T\big|\, dx_1 \to 0$ as $T\to \infty$, so that thanks to $\alpha<1$
\be
\label{ab2}
\|\partial_1 f_T\| = \|\int_T^\infty \frac d{ds} (\partial_1 f)_s\, ds\|\stackrel{\eqref{eq:loi_grupG}}{\leq} \int_T^\infty \|\partial^2_1 u_s\|\, ds\lesssim 
\frac1{T^{1-\alpha}}[u]_\alpha.
\ee

We now show \eqref{rectang} by the following argument: for $x_1, x_1'\in [0,1)$, we write
\[  f(x_1')- f(x_1) = ( f- f_T)(x'_1)+( f_T- f)(x_1) + \int_0^1 (x_1'-x_1) \partial_1  f_T(sx_1'+(1-s)x_1) \, ds. \]
Hence $ \lvert  f(x'_1)- f(x_1) \rvert \leq 2\lVert  f- f_T \rVert + \lVert \partial_1  f_T \rVert \lvert x_1'-x_1 \rvert 
\stackrel{\eqref{ab1}, \eqref{ab2}}{\lesssim}  ( T^\alpha + T^{\alpha-1} |x_1-x_1'|) [u]_\alpha$. Choosing $T$ such that $T=|x_1-x_1'|$ yields 
$[ f]_\alpha \lesssim [u]_\alpha$. 

$\bullet$ If $\alpha \in (1, \frac 32)$, let $u\in C^{\alpha}$ with $\int_0^1 u\, \, dx_1=0$ and $f=Ru$. Then by the above argument, $\partial_1 f=R(\partial_1 u)\in C^{\alpha-1}$ and $[\partial_1 f]_{\alpha-1}\lesssim [\partial_1 u]_{\alpha-1}\lesssim [u]_\alpha$ (by Lemma \ref{lem:linftyhoelder}). Therefore,
$|f(x_1')-f(x_1)-\partial_1 f(x_1)(x'_1-x_1)|\leq |x'_1-x_1| \int_0^1 |\partial_1 f(x_1+s(x_1'-x_1))-\partial_1 f(x_1)|\, ds\lesssim [\partial_1 f]_{\alpha-1} |x_1'-x_1|^{\alpha}$,
and we conclude that $[f]_\alpha\lesssim [u]_\alpha$.

\medskip

\nd {\it Step 2. For $0<\beta<\alpha <\frac32$, the Hilbert transform $R$ satisfies 
$[Ru]_\beta\lesssim [u]_\alpha$ for all periodic functions $u$ of vanishing average in $x_1$, a space we denote by $\FF$ for abbreviation. }

$\bullet$ If $\alpha\in (0,1]$, let $u\in C^\alpha\cap \FF$ and set $f:=Ru$. {As $\beta<\alpha$, we may  w.l.o.g. assume that $\alpha<1$ (otherwise, replace $\alpha$ by $\tilde \alpha:=(\alpha+\beta)/2\in (0,1)$ and use that $[u]_\alpha\gtrsim [u]_{\tilde \alpha}$).}

\nd By Step 1 and Lemma \ref{lem:linftyhoelder}, we know that for every $x_2\in [0,1)$, $f(\cdot, x_2)\in C^{\alpha}([0,1))\cap \FF$ with $[f(\cdot, x_2)]_{\beta}\lesssim [f(\cdot, x_2)]_{\alpha}\lesssim [u(\cdot, x_2)]_{\alpha}\lesssim [u]_{\alpha}$ as $\beta\in (0, \alpha)$; moreover, we have thanks to the vanishing average in $x_1$: $\|f(\cdot, x_2)-f(\cdot, y_2)\|\lesssim [f(\cdot, x_2)-f(\cdot, y_2)]_{\alpha-\beta}\lesssim [u(\cdot, x_2)-u(\cdot, y_2)]_{\alpha-\beta}$,
for every $x_2, y_2\in (0,1)$. To conclude, it is enough to bound the latter RHS by $|x_2-y_2|^{2\beta/3}$. Indeed, by Definition \ref{definitionCalpha},
we have for every $x_1, x_2, y_1, y_2\in [0,1)$: \footnote{We use $\min\{a,b\}\leq a^\eps b^{1-\eps}$ for $\eps\in (0,1)$ and $a,b\geq 0$.}
\begin{align}
\nonumber
\bigg|\big(u(x_1, x_2)-u(x_1, y_2)\big)-\big(u(y_1, x_2)-u(y_1, y_2)\big)\bigg|&\leq 
2 [u]_\alpha \min\{|x_1-y_1|^\alpha, |x_2-y_2|^{2\alpha/3}\}\\
\label{ad34}
&\leq 2 [u]_\alpha |x_1-y_1|^{\alpha-\beta} |x_2-y_2|^{2\beta/3}
\end{align}
yielding $[u(\cdot, x_2)-u(\cdot, y_2)]_{C^{\alpha-\beta}(I)}\lesssim [u]_\alpha |x_2-y_2|^{2\beta/3}$; thus, $[f]_\beta\lesssim [u]_\alpha$.

$\bullet$ If $\alpha\in (1,\frac 3 2)$, let $u\in C^\alpha\cap \FF$ and $f:=Ru$. 
By Step 1, we know $[f(\cdot, x_2)]_{\beta}\lesssim [u(\cdot, x_2)]_{\alpha}\lesssim [u]_\alpha$ for every $x_2\in [0,1)$. Since the function $u(\cdot, x_2)-u(\cdot, y_2)$ is Lipschitz (by Lemma \ref{lem:linftyhoelder}), using the same argument as in \eqref{ad34} and Step~1, we have for every  $x_1, x_2, y_2\in (0,1)$: $|f(x_1, x_2)-f(x_1, y_2)|\lesssim [f(\cdot, x_2)-f(\cdot, y_2)]_{(\alpha-\beta)/\alpha}\lesssim [u(\cdot, x_2)-u(\cdot, y_2)]_{(\alpha-\beta)/\alpha} \lesssim  [u]_{\alpha} |x_2-y_2|^{2\beta/3}$ because $\min\{|x_1-y_1|, |x_2-y_2|^{2\alpha/3}\}\leq |x_1-y_1|^{(\alpha-\beta)/\alpha} |x_2-y_2|^{2\beta/3}$. The conclusion is now straightforward.

\end{proof}


\section{Proof of the main results: Theorems \ref{T1} and \ref{T2}}
\label{sec:main}

The twin Theorems \ref{T1} and \ref{T2} are an immediate consequence of the following purely deterministic
result, which relies on a fixed point argument based on the Schauder theory of
Lemma~\ref{L2} and the regular product result of Lemma~\ref{L3}. 

\medskip

\begin{prop}\label{Pfix}
For given $0<\epsilon<\frac{1}{8}$, there exists a (possibly large) constant $C>0$ with the following property:
Suppose we are given a function $v$ {of vanishing average} and a distribution $F$, both periodic, and small in the sense of
\begin{align}\label{fi01}
[F]_{-\frac{3}{4}-\epsilon}+[v]_{\frac{3}{4}-\epsilon}\le\frac{1}{C}.
\end{align}
Then there exists a unique periodic function $w$ of vanishing average in $x_1$ with
\begin{align*}
[w]_{\frac{5}{4}-2\epsilon}\le\frac{1}{C} 
\end{align*}
and that satisfies (in a distributional sense)
\begin{align*}
\lefteqn{(-\partial_1^2-|\partial_1|^{-1}\partial_2^2)w}\nonumber\\
&+P\Big(F+v\partial_2Rw+w\partial_2Rv+w\partial_2Rw+\partial_2\frac{1}{2}R(w+v)^2-(w+v) \partial_1\frac{1}{2}R(w+v)^2\Big)=0.
\end{align*}
Moreover, we have the a priori estimate
\begin{align}\label{fi16}
[w]_{\frac{5}{4}-2\epsilon}\lesssim [F]_{-\frac{3}{4}-\epsilon}+[v]_{\frac{3}{4}-\epsilon}^2.
\end{align}
Finally, if $\tilde w$ denotes the solution for another data pair $(\tilde v,\tilde F)$, then we have
\begin{align}\label{fi02}
[w-\tilde w]_{\frac{5}{4}-2\epsilon}
\lesssim [F-\tilde F]_{-\frac{3}{4}-\epsilon}+[v-\tilde v]_{\frac{3}{4}-\epsilon}.
\end{align}
\end{prop}

We postpone for the moment the proof of Proposition \ref{Pfix}, and we use it in order to prove our main results:

\begin{proof}[Proof of Theorems \ref{T1} and \ref{T2}]
Fix $0<\epsilon<\frac{1}{8}$. From the stochastic Lemmas \ref{L1} and \ref{L5}, and using the Schauder theory of Lemma \ref{L2} 
in order to pass from $\xi$ to $v$, we know that the random variable
\begin{align*}
\sup_{\ell\le 1}\Big(\big([v_\ell]_{\frac{3}{4}-\epsilon}+[F^\ell]^{\frac12}_{-\frac{3}{4}-\epsilon}\big)
+\ell^{-\frac{\epsilon}{2}}\big([v_\ell-v]_{\frac{3}{4}-\epsilon}+[F^\ell-F]^{\frac12}_{-\frac{3}{4}-\epsilon}\big)\Big)
\end{align*}
has bounded moments of all order $p$. Hence there exists a random variable $\sigma_0\ge 0$ that is almost surely positive such that
on the one hand, we have for all $\ell\le 1$
\begin{align}
\sigma_0 [v_\ell]_{\frac{3}{4}-\epsilon}+\sigma_0^2 [F^\ell]_{-\frac{3}{4}-\epsilon}&\le\frac{1}{C},\label{fi15}\\
\sigma_0 [v_\ell-v]_{\frac{3}{4}-\epsilon}+\sigma_0^2[F^\ell-F]_{-\frac{3}{4}-\epsilon}
&\le\ell^{\frac{\epsilon}{2}},\label{fi17}
\end{align}
where $C$ denotes the constant in Proposition \ref{Pfix}, and on the other hand, $\frac{1}{\sigma_0}$ has bounded moments
of all order. 

\medskip

In view of (\ref{fi15}), for fixed $0 \leq \sigma\le\sigma_0$ and $0<\ell\le 1$
we may apply Proposition \ref{Pfix} with $(\sigma^2 F^\ell,\sigma v_\ell)$ playing the role of $(F,v)$. 
It yields a unique periodic function $w^\ell$ of vanishing average in $x_1$
with $[w^\ell]_{\frac{5}{4}-2\epsilon}\le\frac{1}{C}$ and such that
\begin{align*}
\lefteqn{(-\partial_1^2-|\partial_1|^{-1}\partial_2^2)w^\ell}\nonumber\\
&+P\Big(\sigma^2F^\ell+\sigma v_\ell\partial_2Rw^\ell+\sigma w^\ell\partial_2Rv_\ell+w^\ell\partial_2Rw^\ell\nonumber\\
&+\partial_2\frac{1}{2}R(w^\ell+\sigma v_\ell)^2-(w^\ell+\sigma v_\ell) \partial_1\frac{1}{2}R(w^\ell+\sigma v_\ell)^2\Big)=0.
\end{align*}
Since the convolution parameter $\ell>0$ is present, 
we have by definition $F^\ell=v_\ell\partial_2Rv_\ell$ so that in terms of $u^\ell:=\sigma v_\ell+w^\ell$
and by definition of $v$ through $(-\partial_1^2-|\partial_1|^{-1}\partial_2^2)v=P\xi$, the above equation turns
into the desired Euler-Lagrange equation
\begin{align*}
\lefteqn{(-\partial_1^2-|\partial_1|^{-1}\partial_2^2)u^\ell}\nonumber\\
&+P\Big(u^\ell\partial_2Ru^\ell+\partial_2\frac{1}{2}R(u^\ell)^2-u^\ell \partial_1\frac{1}{2}R(u^\ell)^2-\sigma\xi_\ell\Big)=0.
\end{align*}
The a priori estimate (\ref{fi16}) on $w^\ell$ turns into the desired $[u^\ell-\sigma v_\ell]_{\frac{5}{4}-2\epsilon}$ 
$\lesssim [\sigma^2F^\ell]_{-\frac{3}{4}-\epsilon}+[\sigma v_\ell]_{\frac{3}{4}-\epsilon}^2$ $\lesssim \left(\frac\sigma{\sigma_0}\right)^2$,
where we used (\ref{fi15}) in the last estimate. 

\medskip

We now turn to the first convergence statement as $\ell\downarrow 0$ in (\ref{fi18}), which assumes the form
$\lim_{\ell\downarrow0}[w^\ell-w]_{\frac{5}{4}-2\epsilon}=0$, where $w$ is the solution provided by Proposition \ref{Pfix}
for $(\sigma^2 F,\sigma v)$ playing the role of $(F,v)$ there. In particular by \eqref{fi16} and \eqref{fi15} for $\ell=0$, we obtain \eqref{w09}.
It follows from the convergence (\ref{fi17})
of the data, which we need for $\sigma\le\sigma_0$ but only in the qualitative form of
\begin{align*}
\lim_{\ell\downarrow0}\big([\sigma v_\ell-\sigma v]_{\frac{3}{4}-\epsilon}+[\sigma^2 F^\ell-\sigma^2 F]_{-\frac{3}{4}-\epsilon}\big)=0,
\end{align*}
and the continuity property (\ref{fi02}) of the fixed point $w^\ell$ in the data $(\sigma^2F^\ell,\sigma v_\ell)$.
The second convergence statement in (\ref{fi18}) is contained in (\ref{fi17}) via Lemma \ref{L2}.
\end{proof}

\bigskip

\begin{proof}[Proof of Proposition \ref{Pfix}]

We will apply Banach's contraction mapping theorem on the application
\begin{align*}
(F,v,w)\mapsto \Phi(F,v,w):=&-{(-\partial_1^2-|\partial_1|^{-1}\partial_2^2)}^{-1}{P}\Big(F+v\partial_2Rw+w\partial_2Rv+w\partial_2Rw\nonumber\\
&\quad \quad \quad \quad +\partial_2\frac{1}{2}R(w+v)^2-(w+v) \partial_1\frac{1}{2}R(w+v)^2)\Big);
\end{align*}
we are interested in its fixed points $w=\Phi(F,v,w)$.
We make the standing assumptions
\begin{align}\label{fi06}
[F]_{-\frac{3}{4}-\epsilon},[v]_{\frac{3}{4}-\epsilon},[w]_{\frac{5}{4}-2\epsilon}\le 1
\end{align}
with $v$ and $w$ of vanishing average.
Existence and uniqueness under the smallness condition (\ref{fi01}) will follow from 
the following boundedness and Lipschitz continuity of $\Phi$ in the $w$-variable
\begin{align}
[\Phi(F,v,w)]_{\frac{5}{4}-2\epsilon}
&\lesssim [F]_{-\frac{3}{4}-\epsilon}
+([v]_{\frac{3}{4}-\epsilon}+[w]_{\frac{5}{4}-2\epsilon})^2,\label{fi05}\\
[\Phi(F,v,w)-\Phi(F,v,w')]_{\frac{5}{4}-2\epsilon}
&\lesssim ([v]_{\frac{3}{4}-\epsilon}+[w]_{\frac{5}{4}-2\epsilon}+{[w']_{\frac{5}{4}-2\epsilon}})[w-w']_{\frac{5}{4}-2\epsilon}.\label{fi03}
\end{align}
Indeed, the first property ensures that under the smallness condition (\ref{fi01}) on $(F,v)$,
$w\mapsto \Phi(F,v,w)$ is a self-map on a sufficiently small ball w.\ r.\ t.\ $[w]_{\frac{5}{4}-2\epsilon}$.
The second property ensures that on such a sufficiently small ball, and perhaps strengthening
the smallness condition on $v$, this map $w\mapsto \Phi(F,v,w)$ is a contraction 
w.\ r.\ t.\ $[w]_{\frac{5}{4}-2\epsilon}$. Since the space of all periodic $w$'s with vanishing
average in $x_1$ is complete when endowed with $[w]_{\frac{5}{4}-2\epsilon}$, we are done.
For the a priori estimate (\ref{fi16}) on the fixed point $w$ 
we write $w=\Phi(F,v,0)$ $+(\Phi(F,v,w)-\Phi(F,v,0))$ and note that by (\ref{fi05}) the first
RHS term is estimated by the desired $[F]_{-\frac{3}{4}-\epsilon}+[v]_{\frac{3}{4}-\epsilon}^2$,
whereas the second RHS term can be absorbed into the l.\ h.\ s.\ by contractivity.
For the continuity (\ref{fi02}) of the fixed point $w=w(F,v)$ in $(F,v)$, it is sufficient to establish the
following Lipschitz continuity of $\Phi$ in the $(F,v)$-variables
\begin{align}\label{fi04}
[\Phi(F,v,w)-\Phi(F',v',w)]_{\frac{5}{4}-2\epsilon}
\lesssim [F-F']_{-\frac{3}{4}-\epsilon}+[v-v']_{\frac{3}{4}-\epsilon}.
\end{align}
Indeed, (\ref{fi02}) follows from writing $w-w'$ $=\Phi(F,v,w)-\Phi(F',v',w')$ $=\Phi(F,v,w)-\Phi(F',v',w)$
$+\Phi(F',v',w)-\Phi(F',v',w')$ and applying (\ref{fi04}) on the first RHS term and appealing to
contractivity to absorb the second RHS term.

\medskip

We now turn to the proof (\ref{fi05}), (\ref{fi03}), and (\ref{fi04}) --- always under the assumption 
(\ref{fi06}). By the Schauder theory of Lemma \ref{L2}, it is sufficient to consider
\begin{align}\label{fi08}
\Psi(v,w):=&v\partial_2Rw+w\partial_2Rv+w\partial_2Rw\nonumber\\
&\quad +\partial_2\frac{1}{2}R(w+v)^2-(w+v) \partial_1\frac{1}{2}R(w+v)^2
\end{align}
so that $\Phi(F,v,w)={\cal L}^{-1}PF+{\cal L}^{-1}P\Psi(v,w)$ and to establish
\begin{align*}
[\Psi(v,w)]_{-\frac{3}{4}-2\epsilon}
&\lesssim
([v]_{\frac{3}{4}-\epsilon}+[w]_{\frac{5}{4}-2\epsilon})^2,\\
[\Psi(v,w)-\Psi(v,w')]_{-\frac{3}{4}-2\epsilon}
&\lesssim ([v]_{\frac{3}{4}-\epsilon}+[w]_{\frac{5}{4}-2\epsilon}+{[w']_{\frac{5}{4}-2\epsilon}})[w-w']_{\frac{5}{4}-2\epsilon},\\
[\Psi(v,w)-\Psi(v',w)]_{-\frac{3}{4}-2\epsilon}
&\lesssim[v-v']_{\frac{3}{4}-\epsilon}.
\end{align*}
It is convenient to separate $\Psi$ into a quadratic and a cubic part
so that it is sufficient (also using the ordering of negative exponent
$[\cdot]_{-\frac{3}{4}-2\epsilon}$ $\lesssim [\cdot]_{-\frac{1}{4}-3\epsilon}$ $\lesssim [\cdot]_{-\frac{1}{4}-2\epsilon}$
and positive exponent $[\cdot]_{\frac{3}{4}-2\epsilon}$ $\lesssim[\cdot]_{\frac{3}{4}-\epsilon}$ $\lesssim[\cdot]_{\frac{5}{4}-2\epsilon}$
H\"older norms, see Remark \ref{rem:hoelder_inclu}),
to show the multi-linear estimates for $v_1, v_2, v_3$ of vanishing average:
\begin{align}
[v_1\partial_2Rv_2]_{-\frac{1}{4}-3\epsilon}
&\lesssim [v_1]_{\frac{3}{4}-\epsilon}[v_2]_{\frac{5}{4}-2\epsilon},\label{fi10}\\
[v_1\partial_2Rv_2]_{-\frac{3}{4}-2\epsilon}
&\lesssim [v_1]_{\frac{5}{4}-2\epsilon}[v_2]_{\frac{3}{4}-\epsilon},\label{fi11}\\
[\partial_2R(v_1v_2)]_{-\frac{3}{4}-2\epsilon}
&\lesssim [v_1]_{\frac{3}{4}-\epsilon}[v_2]_{\frac{3}{4}-\epsilon},\label{fi12}\\
[v_1 \partial_1R(v_2v_3)]_{-\frac{1}{4}-2\epsilon}
&\lesssim [v_1]_{\frac{3}{4}-\epsilon}[v_2]_{\frac{3}{4}-\epsilon}[v_3]_{\frac{3}{4}-\epsilon}.\label{fi13}
\end{align}
This is easy for (\ref{fi12}): For periodic functions with vanishing average we have
the algebra property
$[v_1 v_2]_{\frac{3}{4}-\epsilon}$ $\lesssim[v_1]_{\frac{3}{4}-\epsilon}[v_2]_{\frac{3}{4}-\epsilon}$,
cf. Lemma \ref{lem:linftyhoelder} in Appendix;
by Lemma \ref{L4}, we may get rid of the Hilbert transform $R$ at the prize of an $\epsilon$: $[R(v_1v_2)]_{\frac{3}{4}-2\epsilon}$ 
$\lesssim[v_1v_2]_{\frac{3}{4}-\epsilon}$; 
finally, by the definition 
\ref{definitionCbeta} of negative exponent H\"older norms we have $[\partial_2R(v_1v_2)]_{-\frac{3}{4}-2\epsilon}\le[R(v_1v_2)]_{\frac{3}{4}-2\epsilon}$.\\

For (\ref{fi13}) we use the same strategy: we note that by the algebra property $[v_2v_3]_{\frac{3}{4}-\epsilon}$
$\lesssim[v_2]_{\frac{3}{4}-\epsilon}[v_3]_{\frac{3}{4}-\epsilon}$, that by the
boundedness of the Hilbert transform
$[R(v_2v_3)]_{\frac{3}{4}-2\epsilon}$ $\lesssim [v_2v_3]_{\frac{3}{4}-\epsilon}$,
and that by definition of the negative exponent H\"older norms
$[\partial_1R(v_2v_3)]_{-\frac{1}{4}-2\epsilon}$ $\lesssim [R(v_2v_3)]_{\frac{3}{4}-2\epsilon}$.
As a new element, we need to appeal to Lemma \ref{L3} to obtain
$[v_1\partial_1R(v_2v_3)]_{-\frac{1}{4}-2\epsilon}$ $\lesssim[v_1]_{\frac{3}{4}-\epsilon}[\partial_1R(v_2v_3)]_{-\frac{1}{4}-2\epsilon}$
(recall that $v_1$ is of vanishing average),
which requires
$(\frac{3}{4}-2\epsilon)+(-\frac{1}{4}-2\epsilon)$ $=\frac{1}{2}-4\epsilon>0$.\\
Estimates (\ref{fi10}) and (\ref{fi11}) use the same ingredients
\begin{align*}
[v_1\partial_2Rv_2]_{-\frac{1}{4}-3\epsilon}&\lesssim
[v_1]_{\frac{3}{4}-\epsilon}[\partial_2Rv_2]_{-\frac{1}{4}-3\epsilon}\quad\mbox{for}\;(\frac{3}{4}-\epsilon)+(-\frac{1}{4}-3\epsilon)>0\\
&\lesssim[v_1]_{\frac{3}{4}-\epsilon}[Rv_2]_{\frac{5}{4}-3\epsilon}
\lesssim[v_1]_{\frac{3}{4}-\epsilon}[v_2]_{\frac{5}{4}-2\epsilon},\\
[v_1\partial_2Rv_2]_{-\frac{3}{4}-2\epsilon}&\lesssim
[v_1]_{\frac{5}{4}-2\epsilon}[\partial_2Rv_2]_{-\frac{3}{4}-2\epsilon}\quad\mbox{for}\;(\frac{5}{4}-2\epsilon)+(-\frac{3}{4}-2\epsilon)>0\\
&\lesssim[v_1]_{\frac{5}{4}-2\epsilon}[Rv_2]_{\frac{3}{4}-2\epsilon}
\lesssim[v_1]_{\frac{5}{4}-\epsilon}[v_2]_{\frac{3}{4}-\epsilon}.
\end{align*}

\end{proof}


\section{Estimates of the stochastic terms}
\label{sec:sto}

\subsection{Estimate of the white noise. Proof of Lemma \ref{L1}}

\begin{proof}[Proof of Lemma \ref{L1}] To simplify the notation, we will denote $\xi:=P\xi$. We divide the proof in several steps:

\nd {\it Step 1}. In this step, we consider the Fourier coefficients 
$\xi_{\ell}(k):=\int_{[0,1)^2}e^{-ik\cdot x}\xi_{\ell}(x)\, dx$, $k\in(2\pi\mathbb{Z})^2$, of $\xi_\ell=\xi*\phi_\ell$ 
and its logarithmic derivative $\ell\frac{\partial}{\partial\ell}\xi_{\ell}(k)$ in the convolution scale $\ell$.
For $k\not=0$, we claim the following stochastic (second moment) bounds
$$
\langle|\xi_\ell(k)|^2\rangle \lesssim 1, \quad
\langle|\ell\frac{\partial}{\partial\ell}\xi_\ell(k)|^2\rangle\lesssim \min\{1, \ell^2 d^2(k,0)\}, \quad \forall k\in (2\pi \Z)^2.
$$
Recall that $\lesssim$ means that the (generic) constant only depends on $\phi$ in this context.

\medskip

Within our identification $\xi:=P\xi$, $\xi_\ell(k)$ vanishes for $k_1=0$, so that we may restrict to
$k\not=0$. Recall 
$$
{\xi_\ell(k):=\phi(\ell k_1,\ell^\frac{3}{2}k_2)\xi(k),}
$$
and where $\phi(k)$ denotes the Fourier {\it transform} of the Schwartz mask $\phi$ of the
convolution kernel 
$\phi_\ell(x_1,x_2)=\ell^{-\frac{5}{2}}\phi(\frac{x_1}{\ell},\frac{x_2}{\ell^\frac{3}{2}})$, $x\in \R^2$ so that in Fourier space, $\phi_\ell(k)=\phi(\ell k_1, \ell^{3/2} k_2)$, $k\in \R^2$.
We are also interested in the sensitivities $\ell\frac{\partial}{\partial \ell}$ with respect to the
convolution length $\ell$. It is convenient to consider the derivative in this logarithmic
form $\ell\frac{\partial}{\partial \ell}=\frac{\partial}{\partial\ln\ell}$ since it preserves the structure of the convolution:
\begin{equation}
\label{deltaphi}
\ell\frac{\partial}{\partial \ell}[\phi(\ell k_1,\ell^\frac{3}{2} k_2)]
=\ell k_1\frac{\partial\phi}{\partial k_1}(\ell k_1,\ell^\frac{3}{2} k_2)
+\frac{3}{2}\ell^\frac{3}{2} k_2\frac{\partial\phi}{\partial k_2}(\ell k_1,\ell^\frac{3}{2} k_2)=:\delta\phi_\ell(k)
\end{equation}
is the rescaled Fourier transform of another Schwartz function $\del\phi$
given through $\del\phi:$ $=(k_1\frac{\partial}{\partial k_1}+\frac{3}{2}k_2\frac{\partial}{\partial k_2})\phi$,
which in real space assumes the form
$\del\phi=-\partial_1(x_1\phi)-\frac{3}{2}\partial_2(x_2\phi)$ 
$=-(\frac{5}{2}+x_1\partial_1+\frac{3}{2}x_2\partial_2)\phi$. Therefore, this prompts the definition of
$$
\del \xi_\ell(k):=\ell \frac{\partial}{\partial \ell} \xi_\ell(k)=\del\phi(\ell k_1,\ell^\frac{3}{2} k_2)\xi(k).
$$
By the relation between convolution and Fourier series (as explained at the beginning of Section \ref{sec:holder_space}) we have $\xi_\ell(k)=\phi_\ell(k) \xi(k)$ where $\xi(k)$ is the Fourier coefficient of $\xi$. By the characterizing property of white noise we have
$$\langle|\xi(k)|^2\rangle=\int_{[0,1)^2} dx \int_{[0,1)^2} dx' e^{i k\cdot (x'-x)} \langle\xi(x) \bar \xi(x')\rangle=\int_{[0,1)^2} dx=1. $$
Hence, we deduce the desired estimates:%
$$
\langle|\xi_\ell(k)|^2\rangle=|\phi_\ell|^2(k)\lesssim 1 \quad 
\langle|\ell\frac{\partial}{\partial \ell} \xi_\ell(k)|^2\rangle=|\del \phi_\ell|^2(k)\lesssim \min\{1, \ell^2 d^2(k,0)\}, \, k\neq 0,
$$
where we used that $|\delta\phi(k)|\leq \min\{1, d(k,0)\}$ for all $k\in \R^2$ (because by definition, $\del\phi$ vanishes for $k=0$).

\medskip

\nd {\it Step 2}. We claim that for all $\ell\le 1$, $x\in\mathbb{R}^2$, and $T>0$ we have the estimate
\begin{align}
\langle(\xi_{\ell,T}(x))^2\rangle^\frac{1}{2}
&\lesssim(T^\frac{1}{3})^{-\frac{5}{4}},\label{f18old}\\
\langle(\ell\frac{\partial}{\partial \ell} \xi_{\ell,T}(x))^2\rangle^\frac{1}{2}
&\lesssim \min\{(T^\frac{1}{3})^{-\frac{5}{4}},\ell(T^\frac{1}{3})^{-\frac{9}{4}}\} ,\label{f19old}
\end{align} 
where $\xi_{\ell, T}:=\xi*\phi_\ell*\psi_T$ and $\psi_T$ is given by \eqref{eq14bis}.
Indeed, since the distribution $\xi$ and its translation $\xi(\cdot+h)$ by some translation vector $h$ have
the same distribution under $\langle\cdot\rangle$, this shift-invariance carries over to $\xi_\ell$ and 
$\xi_{\ell,T}$. This implies that 
$\langle(\xi_{\ell,T}(x))^2\rangle$ does not depend on $x$. 
Hence for (\ref{f18old}) it is enough to establish the space-integrated version
\begin{align*}
\langle\int_{[0,1)^2}\xi_{\ell,T}^2\, dx\rangle
\lesssim(T^\frac{1}{3})^{-\frac{5}{2}}.
\end{align*}
In conjunction with the periodicity of $\xi_{\ell,T}=\psi_T*\xi_\ell$, this allows us to appeal to Plancherel
and the relation between convolution and Fourier series (see the beginning of Section \ref{sec:holder_space}),
into which we insert the Fourier characterization (\ref{eq14bis}) of $\psi_T$:
$$
\langle \int_{[0,1)^2}\xi_{\ell,T}^2 \, dx \rangle
=\sum_{k\in(2\pi\mathbb{Z})^2} \exp(-2T(|k_1|^3+k_2^2))
\langle|\xi_\ell(k)|^2 \rangle \lesssim \sum_{k\in(2\pi\mathbb{Z})^2, \, k\neq 0} \exp(-T(|k_1|^3+k_2^2))
$$
where we used Step 1 (and that $\xi_\ell(k=0)=\xi(k=0)=0$ thanks to the presence of $P$ in the
definition of $\xi:=P\xi$). It is thus sufficient to show:
\begin{align}\label{f20old}
\sum_{k\in (2\pi \Z)^2\setminus\{0\}}\exp(-T d^3(k,0))\lesssim(T^\frac{1}{3})^{-\frac{5}{2}}, \quad \forall T>0.
\end{align}
To prove this kind of estimate, we will systematically use the following algorithm: by the obvious ``volume scaling''
\begin{align}\label{f16old}
\#\{k\in (2\pi \Z)^2 \, | \, \ell d(k,0)\le 1\}\lesssim \ell^{-\frac{5}{2}}, \quad \ell\leq 1,
\end{align}
we do a decomposition into dyadic annuli (based on the distance $d$), so that (\ref{f16old}) implies
that the above integral may be estimated as their Euclidean counterpart in dimension $\frac{5}{2}$:
\begin{align*}
\sum_{k\in (2\pi \Z)^2\setminus\{0\}}\exp(-T d^3(k,0))&\leq \sum_{n\in \Z} \quad \sum_{T^{\frac13} d(k,0)\in (2^{n-1}, 2^n], \, k\in (2\pi \Z)^2} \exp(-2^{3(n-1)}) \\
& \leq \sum_{n\in \Z}  \exp({-2^{3(n-1)}})\#\{k\not=0\, | \, T^{\frac13} d(k,0)\le 2^{n}\}\\
& \lesssim (T^{-\frac13})^{\frac 52} \sum_{n\in \Z} 2^{\frac{5n}2} \exp({-2^{3(n-1)}})\lesssim (T^{-\frac13})^{\frac 52}.
\end{align*}
Turning to (\ref{f19old}), which differs from (\ref{f18old}) through the presence of $\ell\frac{\partial}{\partial \ell}$, 
we see by analogous arguments that we need to establish (due to Step 1):
$$
\sum_{k\in (2\pi \Z)^2\setminus\{0\}}\exp(-Td^3(k,0))\min\{1,\ell^2 d^2(k,0)\}
\lesssim\min\{(T^\frac{1}{3})^{-\frac{5}{2}},\ell^2(T^\frac{1}{3})^{-\frac{9}{2}}\}, \quad \forall \ell\leq 1, T>0,
$$
which splits into (\ref{f20old}) and 
\begin{align}\label{f21old}
\sum_{k\not=0}\exp(-Td^3(k,0))d^2(k,0)=T^{-\frac23}\sum_{k\not=0}\exp(-Td^3(k,0))\big(T^{\frac13}d(k,0)\big)^2\lesssim(T^\frac{1}{3})^{-\frac{9},{2}}
\end{align}
which follows by the same argument as above.

\medskip

\nd {\it Step 3}.
We claim that for all $\ell\le 1$, $1\leq p<\infty$, and $T>0$ we have the estimate
\begin{align}
\langle\|\xi_{\ell,T}\|^p\rangle^\frac{1}{p}
&\lesssim(T^\frac{1}{3})^{-\frac{5}{4}-\frac{1}{p}\frac{5}{2}},\label{f22old}\\
\langle\|\ell\frac{\partial}{\partial \ell} \xi_{\ell,T}\|^p\rangle^\frac{1}{p}
&\lesssim \min\{(T^\frac{1}{3})^{-\frac{5}{4}-\frac{1}{p}\frac{5}{2}},\ell(T^\frac{1}{3})^{-\frac{9}{4}-\frac{1}{p}\frac{5}{2}}\} .\label{f23old}
\end{align}  
Note that both the left-hand side (LHS) becomes larger (Jensen's) and the RHS smaller (at least for $T\le 1$)
as $p$ increases; however, the constant hidden in $\lesssim$ now depends on $p$ and blows
up as $p\uparrow\infty$.
Estimates (\ref{f22old}) and (\ref{f23old}) follow from Step 2 in two stages. 
Here comes the first, stochastic stage:
The two random variables considered in Step 2, i.e.,
$f \in\{\xi_{\ell,T}(x), \ell\frac{\partial}{\partial \ell} \xi_{\ell,T}(x)\}$, are linear expressions in the Gaussian field $\xi$ and thus centered Gaussian random variables. As such, they
satisfy an inverse Jensen's inequality 
$\langle |f|^p\rangle^\frac{1}{p}\lesssim\langle f^2\rangle^\frac{1}{2}$ for all $1\leq p<\infty$. Hence
we obtain from Step 2, after integration in $x$
$$
\langle\int_{[0,1)^2}|\xi_{\ell,T}|^p\rangle^\frac{1}{p}
\lesssim(T^\frac{1}{3})^{-\frac{5}{4}}, \quad \langle\int_{[0,1)^2}|\ell\frac{\partial}{\partial \ell} \xi_{\ell,T}|^p\rangle^\frac{1}{p}
\lesssim \min\{(T^\frac{1}{3})^{-\frac{5}{4}}, \ell (T^\frac{1}{3})^{-\frac{9}{4}}\}.
$$
We now turn to the second, deterministic stage. It follows from the fact that for all the fields
$f$ $\in\{\xi_\ell,$ $\ell\frac{\partial}{\partial \ell} \xi_\ell\}$ we have
by the semi-group property (\ref{f24}) in form of $f_T=\psi_\frac{T}{2}*f_\frac{T}{2}$ so that for $p\in (1, \infty)$:
\begin{align*}
\|f_T\|&\le
\Big(\int_{\R^2}|\psi_\frac{T}{2}|^\frac{p}{p-1}\Big)^\frac{p-1}{p}
\Big(\int_{[0,1)^2}|f_\frac{T}{2}|^p                              \Big)^\frac{1}{p}\nonumber\\
\stackrel{(\ref{f25})}{=}&\big((T^\frac{1}{3})^\frac{5}{2}\big)^{-\frac{1}{p}}
\Big(\int_{\R^2}|\psi|^\frac{p}{p-1}\Big)^\frac{p-1}{p}
\Big(\int_{[0,1)^2}|f_\frac{T}{2}|^p\Big)^\frac{1}{p},
\end{align*}
where the ``volume'' factor $(T^\frac{1}{3})^\frac{5}{2}$ arises from the rescaling of $\, dx$.
The combination of $\psi\in L^1$ (see (\ref{e03})) with $\|\psi\|\lesssim \int |\psi(k)|\, dk<\infty$ yields 
$(\int_{\R^2}|\psi|^\frac{p}{p-1})^\frac{p-1}{p}$ $\lesssim 1$, so that the above turns into
the ``inverse estimate''
\begin{align}\label{f35old}
\|f_T\|\lesssim(T^\frac{1}{3})^{-\frac52\frac{1}{p}}\Big(\int_{[0,1)^2}|f_\frac{T}{2}|^p\Big)^\frac{1}{p}.
\end{align}
This proves \eqref{f22old}-\eqref{f23old} for $p>1$. For $p=1$, the conclusion follows by Jensen's inequality.

\medskip

\nd {\it Step 4}. We claim that for all $\eps\in (0,1)$, $1\leq  p<\infty$ and $\ell_0\le 1$ we have for every $T\leq 1$:
\begin{align}
\langle\sup_{\ell\le 1}\|\xi_{\ell,T}\|^p\rangle^\frac{1}{p}
&\lesssim  (T^\frac{1}{3})^{-\frac{5}{4}-2\eps},
\label{f32old}\\
\langle\sup_{\ell,\ell'\le \ell_0}\|\xi_{\ell,T}-\xi_{\ell',T}\|^p\rangle^\frac{1}{p}
&\lesssim\ell_0^\eps
(T^\frac{1}{3})^{-\frac{5}{4}-2\eps}, \label{f31old}
\end{align}
where now, the constant also depends on $\eps>0$.
Clearly, by the triangle inequality w.r.t. to $\langle\|\cdot\|^p\rangle^{\frac1p}$, (\ref{f31old}) follows from
\begin{align}\label{f33old}
\langle\big(\int_0^{\ell_0}\|\frac{\partial}{\partial \ell} \xi_{\ell,T}\| d\ell\big)^p\rangle^\frac{1}{p}
&\lesssim\ell_0^\eps
(T^\frac{1}{3})^{-\frac{5}{4}-2\eps},
\end{align}
whereas (\ref{f32old}) then is a consequence of this for $\ell_0=1$ and (\ref{f22old}) for $\ell=1$ provided that $p$ is so large 
that $\frac1p \frac52\leq 2\eps$; then, Jensen's inequality will also lead to 
(\ref{f32old}) for smaller $p$. 
Estimate (\ref{f33old}) in turn follows from
(\ref{f23old}) and the triangle inequality w.r.t. to $\langle|\cdot|^p\rangle^\frac{1}{p}$ in form of
\begin{align*}
\langle\big(\int_0^{\ell_0}\|\frac{\partial}{\partial \ell} \xi_{\ell,T}\|d\ell\big)^p\rangle^\frac{1}{p}
&\le\int_0^{\ell_0}\langle\|\frac{\partial}{\partial \ell} \xi_{\ell,T}\|^p\rangle^\frac{1}{p}d\ell
=\int_0^{\ell_0}\langle\|\ell\frac{\partial}{\partial \ell} \xi_{\ell,T}\|^p\rangle^\frac{1}{p}\frac{d\ell}{\ell}
\end{align*}
and the fact that the RHS of (\ref{f23old}) can be majorated as follows
\begin{align*}
\min\{(T^\frac{1}{3})^{-\frac{5}{4}-\frac{1}{p}\frac{5}{2}},
\ell(T^\frac{1}{3})^{-\frac{9}{4}-\frac{1}{p}\frac{5}{2}}\}\le
\ell^\eps(T^\frac{1}{3})^{-\frac{5}{4}-2\eps}
\end{align*}
because $\min\{a, b\}\leq a^\eps b^{1-\eps}$ for $a,b>0$ and $\eps\in (0,1)$ and for $p\geq 1$ so large that $\frac1p \frac52\leq \eps$. For smaller $p$, the estimate \eqref{f33old} follows by Jensen's inequality.

\medskip

\nd {\it Step 5}. We claim that for every $0<2\eps<\frac{1}{4}$ and any $p\geq 1$:
\begin{align*}
\langle\sup_{\ell\le 1}[\xi_\ell]_{-\frac54-2\eps}^p\rangle^\frac{1}{p}&\lesssim 1,\quad 
\langle\sup_{\ell,\ell'\le\ell_0}[\xi_\ell-\xi_{\ell'}]_{-\frac54-2\eps}^p\rangle^\frac{1}{p}\lesssim \ell_0^\eps, \quad \ell_0\leq 1.
\end{align*}
First, by Jensen's inequality in $\langle\cdot\rangle$, we note that Step 4 may be reformulated as follows: For every $0<2\eps<\frac{1}{4}$, $p\geq 1$, $\ell_0\leq 1$ and any $T\in (0,1]$:
\be
\label{eq:estnoise1}
\langle\sup_{\ell\le 1}\|\xi_{\ell,T}\|^p\rangle^\frac{1}{p}
\lesssim (T^\frac{1}{3})^{-\frac54-2\eps}, \quad
\langle\sup_{\ell,\ell'\le \ell_0}\|\xi_{\ell,T}-\xi_{\ell',T}\|^p\rangle^\frac{1}{p}
\lesssim \ell_0^\eps (T^\frac{1}{3})^{-\frac54-2\eps}.\ee
Let $T_n:=2^{n}$, $n\in \{0,-1,-2, \dots\}$. Note that for $T\in (0,1]$ we choose $T_{n}< T \leq T_{n+1}$ and we have 
\[ \lvert \xi_{\ell,T}(x) \rvert = \lvert (\xi_{\ell, T_{n}})_{T-T_{n}}(x) \rvert \leq \lVert \xi_{\ell, T_{n}} \rVert \int_{\R^2} \lvert \psi_{T-T_{n}}(y) \rvert dy  \lesssim \lVert \xi_{\ell, T_{n}} \rVert\]
as well as $\lvert \xi_{\ell,T}(x)-\xi_{\ell',T}(x) \rvert \lesssim  \lVert \xi_{\ell, T_n}-\xi_{\ell',T_n} \rVert$.
Then we have for some $0<\eps'<\eps$:
\begin{align*}
 \Bigl\langle \sup_{\ell\leq1} \sup_{T\in (0,1]}  \Bigl( (T^{\frac13})^{\frac54+2\eps} \lVert \xi_{\ell, T} \rVert \Bigr)^p \Bigr\rangle 
& \lesssim \Bigl\langle \sup_{\ell\leq1} \sup_{n\in \Z, n\leq 0} \Bigl( (T_n^{\frac13})^{\frac54+2\eps} \lVert \xi_{\ell,T_n} \rVert \Bigr)^p \Bigr\rangle\\
& \leq \Bigl\langle \sup_{\ell\leq1} \sum_{n\in \Z, n\leq 0} \Bigl( (T_n^{\frac13})^{\frac54+2\eps} \lVert \xi_{\ell, T_n} \rVert \Bigr)^p \Bigr\rangle\\
&\leq 
\sum_{n\in \Z, n\leq 0} \Bigl( (T_n^{\frac13})^{{p(\frac54+2\eps)} }\underbrace{\bigl\langle \sup_{\ell\leq1} \lVert \xi_{\ell, T_n} \rVert^p\bigr\rangle}_{\stackrel{\eqref{eq:estnoise1}}{\lesssim} (T_n^{\frac13})^{-(\frac54+2\eps')p}} \Bigr)\lesssim 1.
\end{align*}
By the same argument, relying on \eqref{f31old}, we also obtain:
$$
\Bigl\langle \sup_{\ell, \ell'\leq \ell_0}\sup_{T\in(0,1]} \Bigl( (T^{\frac13})^{{\frac54+2\eps}} \lVert \xi_{\ell,T} {- \xi_{\ell',T}} \rVert \Bigr)^p \Bigr\rangle^{\frac1p}\lesssim \ell_0^\eps.$$
By Lemma \ref{Lequiv} to convert the convolution-based norm
into a negative exponent H\"older norm we obtain the conclusion of Step 5. 

As $\{\xi_{\ell}\}$ is a Cauchy ``sequence" in the Banach space $C^{-\frac54-2\eps}$ we deduce that the limit $\xi$ of  $\xi_{\ell}$ as $\ell\to 0$ satisfies $\Bigl\langle \sup_{\ell\leq \ell_0} [\xi-\xi_\ell]_{-\frac54-2\eps}^p \Bigr\rangle^{1/p}\lesssim \ell_0^\eps$
for every $p\in [1, \infty)$.
\end{proof}


\subsection{Estimate of off-line term. Proof of Lemma \ref{L5}}\label{Off}


\begin{proof}[Proof of Lemma \ref{L5}] We consider the Fourier coefficients 
$F^{\ell}(k):=\int_{[0,1)^2}e^{-ik\cdot x}F^{\ell}(x)\, dx$, $k\in(2\pi\mathbb{Z})^2$, of $F^\ell$ 
and its logarithmic derivative $\ell\frac{\partial}{\partial\ell}F^{\ell}(k)$ in the convolution scale $\ell$.
For $k\in (2\pi \Z)^2\setminus\{0\}$, we claim the following stochastic (second moment) bounds:
\begin{equation}
\label{clai}
\langle|F^\ell(k)|^2\rangle\lesssim d^{-1}(k,0), \, \, 
\langle|\ell\frac{\partial}{\partial\ell}F^\ell(k)|^2\rangle\lesssim\min\{d^{-1}(k,0),\ell^2d(k,0)\}, \, \forall \ell\leq 1.
\end{equation}
Here, $\lesssim$ means up to a (generic) constant that only depends on $\phi$.

\medskip

Here comes the argument: Because of the projection $P$, $F^\ell(k)$ vanishes for $k_1=0$, so that we may restrict to
$k_1\not=0$. We appeal to the formula for the product
\begin{displaymath}
F^\ell(k)=\sum_{k'+k''=k}v_\ell(k')(\partial_2 Rv_\ell)(k'')
=-\sum_{k'+k''=k}v_\ell(k')(\sgn k_1'')k_2'' v_\ell(k''),
\end{displaymath}
where we used that the Fourier multiplier of $R$ is $i\sgn k_1$ and that
of $\partial_2$ is $ik_2$. By definition of $v$ via $Pv=v$ and 
$(-\partial_1^2-|\partial_1|^{-1}\partial_2^2)v=P\xi$ we have on the Fourier level
\begin{align}\nonumber
v_\ell(k)=G_\ell(k)\xi(k), \quad k\in (2\pi \Z)^2,
\end{align}
with the abbreviations \footnote{Do not confound $G$ with the ``heat kernel" used in the proof of Lemma \ref{L4}.}
\begin{align}\label{f8}
G(k):=\frac{|k_1|}{|k_1|^3+k_2^2}\quad\mbox{and}\quad
G_\ell(k):=\phi(\ell k_1,\ell^\frac{3}{2}k_2)G(k);
\end{align}
recall that $\phi(k)$ denotes the Fourier {\it transform} of the {symmetric} Schwartz mask $\phi$ of the
convolution kernel 
$\phi_\ell(x_1,x_2)=\ell^{-\frac{5}{2}}\phi(\frac{x_1}{\ell},\frac{x_2}{\ell^\frac{3}{2}})$. {As $\phi(x)$ is symmetric, the Fourier transform $\phi(k)$ is real valued.} 
With the abbreviations
\begin{align}\label{f1}
\tilde G(k):=-(\sgn k_1)k_2G(k)\quad\mbox{and}\quad 
\tilde G_\ell(k):=
\phi(\ell k_1,\ell^\frac{3}{2} k_2)\tilde G(k)
\end{align}
we thus obtain the formula
\begin{equation}\label{f2}
F^\ell(k)=\sum_{k'+k''=k}G_\ell(k')\tilde G_\ell(k'')\xi(k')\xi(k'').
\end{equation}
Recall the definition of $\del\phi_\ell$ introduced in \eqref{deltaphi} as well as the sensitivities $\ell\frac{\partial}{\partial \ell}$ with respect to the
convolution length $\ell$. In line with the second items in
(\ref{f8}) and (\ref{f1}),
this prompts the definition of
\begin{align}\label{f6}
\del G_\ell(k):=\del\phi(\ell k_1,\ell^\frac{3}{2} k_2)G(k),\quad
\del \tilde G_\ell(k):=\del\phi(\ell k_1,\ell^\frac{3}{2} k_2)\tilde G(k).
\end{align}
Hence from (\ref{f2}), we find by Leibniz' rule
\begin{align}\label{f4}
\ell\frac{\partial}{\partial \ell} F^\ell(k)=\sum_{k'+k''=k}
(\del G_\ell(k')\tilde G_\ell(k'')+G_\ell(k')\del\tilde G_\ell(k''))\xi(k')\xi(k'').
\end{align}

\nd {\it Step 1}. The first step for \eqref{clai} is to prove the following identities for the white noise: for $k',k'',l',l''\not=0$,
\begin{align}\label{f5}
\lefteqn{\langle\xi(k')\xi(k'')\overline{\xi(l')}\,\overline{\xi(l'')}\rangle}\nonumber\\
&=\left\{\begin{array}{cl}
\langle|\xi(k')|^2|\xi(k'')|^2\rangle&\mbox{for}\;\{k',k''\}=\{l',l''\}\\
\langle|\xi(k')|^2|\xi(l')|^2\rangle&\mbox{for}\;k'+k''=l'+l''=0\\
0&\mbox{else}
\end{array}\right\}.
\end{align}
Indeed, it follows easily from the characterization of white noise $\xi$ that the
real-valued random variables in the family $\{{\mathcal Re}\xi(k), {\mathcal Im}\xi(k)\}_{k\not=0}$ 
are centered, of variance $\frac{1}{2}$ and of vanishing covariances. 
Since these variables are also jointly Gaussian, they are in fact independent (and identically distributed) besides the linear constraint 
$\overline{\xi(k)}=\xi(-k)$.
For the sake of completeness, let us give an argument for (\ref{f5}) in form of
\begin{align}\label{f5bis}
\lefteqn{\langle\xi(k^1)\xi(k^2)\xi(k^3)\xi(k^4)\rangle}\nonumber\\
&=\left\{\begin{array}{cl}
\langle|\xi(k^1)|^2|\xi(k^2)|^2\rangle&\mbox{for}\;\{-k^1,-k^2\}=\{k^3,k^4\}\\
0&\mbox{if $\{k^1,k^2,k^3,k^4\}$ is not composed}\\[-0.5ex]
&\mbox{ of two pairs that sum to zero}
\end{array}\right\}.
\end{align}
Note that the RHS of (\ref{f5bis}) does not cover all cases explicitly; the missing cases are implicitly
covered by permutation symmetry of the lhs.
Here comes the argument for (\ref{f5bis}): Since for $k^1,k^2\not=0$, 
$\xi(k^1)=\overline{\xi(-k^1)}$ and $\xi(k^2)=\overline{\xi(-k^2)}$ are
independent unless $|k^1|=|k^2|$, the expression on the lhs of (\ref{f5bis}) vanishes
unless $(k^1,k^2,k^3,k^4)$ is composed by
two pairs of indices that agree up to the sign. By permutation we may w.l.o.g. assume that $\{|k_1|,|k_2|\}=\{|k_3|,|k_4|\}$.
We have to distinguish 4 cases: 1) $\{-k^1,-k^2\}=\{k^3,k^4\}$, in which case the expression turns
into the desired $\langle|\xi(k^1)|^2|\xi(k^2)|^2\rangle$. 2) $\{-k^1,k^2\}=\{k^3,k^4\}$, in which case
the expression turns into $\langle|\xi(k^1)|^2\xi^2(k^2)\rangle$. This expression vanishes, since
shifting $\xi$ to $\xi(\cdot+h)$ by a
shift vector $h$ with $h\cdot k^2=-\frac{\pi}{2}$ does not change the white-noise distribution
but changes the Fourier coefficient $\xi(k^2)$ by a factor of $i$,
while $|\xi(k^1)|^2$ is preserved. 3) $\{k^1,-k^2\}=\{k^3,k^4\}$, in which case we obtain $0$ for the same reason.
4) $\{-k^1,-k^2\}=\{k^3,k^4\}$, in which case we obtain $\langle\xi^2(k^1)\xi^2(k^2)\rangle$; here, we have to distinguish
the three sub cases: 4a) $|k^1|\not=|k^2|$ in which case $\xi^2(k^1)$ and $\xi^2(k^2)$ are
independent so that the expression assumes the form $\langle\xi^2(k^1)\rangle \langle\xi^2(k^2)\rangle$ with both factors
vanishing (by the argument under 2)). 
4b) $-k^1=k^2$, in which $\{k^1,k^2,k^3,k^4\}$ {\it is} composed of two pairs that sum to zero
and thus does not fall under the second case in (\ref{f5bis}). 4c) $k^1=k^2$ in which case the
expression turns into $\langle\xi^4(k^1)\rangle$ which can be seen to vanish by shifting $\xi$ by a vector
$h$ with $h\cdot k^1=-\frac{\pi}{4}$.

\medskip

\nd {\it Step 2}. We prove
\begin{align}
\langle|F^\ell(k)|^2\rangle&\lesssim \sum_{\stackrel{k'+k''=k}{k',k''\not=0}}d^{-4}(k',0)d^{-1}(k'',0),\nonumber\\
\langle|\ell\frac{\partial}{\partial \ell} F^\ell(k)|^2\rangle
&\lesssim\sum_{\stackrel{k'+k''=k}{k',k''\not=0}}
{\min}^2 \{1,\ell(d(k',0)+d(k'',0))\} d^{-4}(k',0)d^{-1}(k',0).\nonumber
\end{align}
For that, we make use of (\ref{f5}); because of $k_1\not=0$, the middle case of (\ref{f5}) does not occur
when applying it to the square of (\ref{f2}):
\begin{align}\nonumber
\langle|F^\ell(k)|^2\rangle=\sum_{k'+k''=k}&\big((G_\ell(k')\tilde G_\ell(k''))^2
+(G_\ell(k')\tilde G_\ell(k''))(G_\ell(k'')\tilde G_\ell(k'))\big)\nonumber\\
&\times\langle|\xi(k')|^2|\xi(k'')|^2\rangle.\nonumber
\end{align}
Because of the same structure, starting from (\ref{f4}) we get
\begin{align}
\langle|\ell\frac{\partial}{\partial \ell} F^\ell(k)|^2\rangle&=\sum_{k'+k''=k}
\bigg((\del G_\ell(k')\tilde G_\ell(k'')+G_\ell(k')\del\tilde G_\ell(k''))^2\nonumber\\
&+(\del G_\ell(k')\tilde G_\ell(k'')+G_\ell(k')\del\tilde G_\ell(k''))\times(\del G_\ell(k'')\tilde G_\ell(k')+G_\ell(k'')\del\tilde G_\ell(k'))\bigg)\nonumber\\
&\times\langle|\xi(k')|^2|\xi(k'')|^2\rangle.\nonumber
\end{align}
With help of Young's inequality (in conjunction with the symmetry
of $\langle|\xi(k')|^2|\xi(k'')|^2\rangle$ in $k'\leftrightarrow k''$) we may simplify to
\begin{align}\nonumber
\langle|F^\ell(k)|^2\rangle\le 2\sum_{k'+k''=k}G_\ell^2(k')\tilde G_\ell^2(k'')
\langle|\xi(k')|^2|\xi(k'')|^2\rangle,
\end{align}
and, applying this argument twice,
\begin{align}
\lefteqn{\langle|\ell\frac{\partial}{\partial \ell} F^\ell(k)|^2\rangle}\nonumber\\
&\le 4\sum_{k'+k''=k}
(\del G_\ell^2(k')\tilde G_\ell^2(k'')+G_\ell^2(k')\del\tilde G_\ell^2(k''))
\langle|\xi(k')|^2|\xi(k'')|^2\rangle.\nonumber
\end{align}
We appeal to the Cauchy-Schwarz inequality in form of $\langle|\xi(k')|^2|\xi(k'')|^2\rangle$
$\le(\langle|\xi(k')|^4\rangle\langle|\xi(k'')|^4\rangle)^\frac{1}{2}$, the identical distribution
of $\{\xi(k)\}_{k\not=0}$ in form of $\langle|\xi(k')|^4\rangle=\langle|\xi(k'')|^4\rangle$,
the independence and identical distribution of ${\mathcal Re}\xi(k)$ and ${\mathcal Im}\xi(k)$ (for $k\not=0$) in form of
$\langle|\xi(k)|^4\rangle$ $=2(\langle({\mathcal Re}\xi(k))^4\rangle$ $+\langle({\mathcal Re}\xi(k))^2\rangle^2)$,
and the (standard) Gaussianity of ${\mathcal Re}\xi(k)$ in form of 
$$\langle({\mathcal Re}\xi(k))^4\rangle=
\big(\frac12\big)^4\langle(\underbrace{2{\mathcal Re}\xi(k)}_{\textrm{of variance 1}})^4\rangle=
\big(\frac12\big)^4 3 \langle({2{\mathcal Re}\xi(k)})^2\rangle=\frac3{16}$$ 
 to conclude
$\langle|\xi(k')|^2|\xi(k'')|^2\rangle\lesssim 1$, which we insert:
\begin{align}
\langle|F^\ell(k)|^2\rangle&\lesssim \sum_{\stackrel{k'+k''=k}{k',k''\not=0}}G_\ell^2(k')\tilde G_\ell^2(k''),\label{f9}\\
\langle|\ell\frac{\partial}{\partial \ell} F^\ell(k)|^2\rangle
&\lesssim\sum_{\stackrel{k'+k''=k}{k',k''\not=0}}
(\del G_\ell^2(k')\tilde G_\ell^2(k'')+G_\ell^2(k')\del\tilde G_\ell^2(k'')).\label{f10}
\end{align}
By the following (build-in) relation between the symbol $G(k)=\frac{|k_1|}{|k_1|^3+k_2^2}$ of ${\mathcal L}^{-1}P$
and the intrinsic metric (which we obtain with help of the Young inequality), namely
\begin{align}\label{f36}
G(k)\lesssim d^{-2}(k,0)\quad\mbox{and thus}\quad |\tilde G(k)|\lesssim d^{-\frac{1}{2}}(k,0),
\end{align}
we obtain from the definitions (\ref{f8}) \& (\ref{f1})
\begin{align}\nonumber
|G_\ell(k)|\lesssim d^{-2}(k,0)\quad\mbox{and thus}\quad|\tilde G_\ell(k)|\lesssim d^{-\frac{1}{2}}(k,0), 
\end{align}
where we used that for our Schwartz kernel $|\phi(k)|\lesssim 1$. By definition, $\del\phi$ vanishes for $k=0$
so that here, we even have $|\del\phi(k)|\lesssim \min\{1,d(k,0)\}$. Using this in the definitions (\ref{f6})
we obtain
$$
|\del G_\ell(k)|\lesssim \min\{1,\ell d(k,0)\}d^{-2}(k,0)\quad\mbox{and} \quad |\del \tilde G_\ell(k)|\lesssim \min\{1,\ell d(k,0)\} 
d^{-\frac{1}{2}}(k,0).$$
Inserting this into (\ref{f9}) and (\ref{f10}) we obtain the conclusion of Step 2.

\medskip

\nd {\it Step 3}. In order to conclude with the claim \eqref{clai}, in view of Step 2, it remains to show
\begin{align}
\sum_{\stackrel{k'+k''=k}{k',k''\not=0}}d^{-4}(k',0)d^{-1}(k'',0)&\lesssim d^{-1}(k,0)\label{newul}
\end{align}
and for $\ell\le 1$ and $k\neq 0$,
\begin{align}
\sum_{\stackrel{k'+k''=k}{k',k''\not=0}}
{\rm min}^2\{1,\ell(d(k',0)+d(k'',0))\}d^{-4}(k',0)d^{-1}(k'',0)
\lesssim\min\{d^{-1}(k,0),\ell^2d(k,0)\}.\label{f11}
\end{align}
We will focus on the more subtle (\ref{f11}). 
Obviously, (\ref{f11}) splits into the three statements
\begin{align}
\sum_{\stackrel{k'+k''=k}{k',k''\not=0}}\frac{1}{d^4(k',0)d(k'',0)}&\lesssim\frac{1}{d(k,0)},\label{wg05}\\
\sum_{\stackrel{k'+k''=k}{k',k''\not=0}}\frac{1}{d^2(k',0)d(k'',0)}&\lesssim d(k,0),\label{wg04}\\
\sum_{\stackrel{k'+k''=k}{k',k''\not=0}}\frac{d(k'',0)}{d^4(k',0)} &\lesssim d(k,0).\label{wg01}
\end{align}
As we shall see, all these statements rely on
\begin{align}\label{wg02}
\sum_{k'\not=0}\frac{1}{d^4(k',0)}\lesssim\sum_{k'\not=0}\frac{1}{d^3(k',0)}\lesssim 1,
\end{align}
which is an immediate consequence of (\ref{f16old}), expressing that the effective
dimension of the $k$-space is $\frac{5}{2}<3$, and the triangle inequality in form of
\begin{align}\label{wg03}
|d(k'',0)-d(k',0)|\le d(k,0)\le d(k',0)+d(k'',0),
\end{align}
as a consequence of $k'+k''=k$.
Indeed, by (\ref{wg03}) in form of $d(k'',0)$ $\le d(k',0)+d(k,0)$ and (\ref{wg02}), 
the LHS of (\ref{wg01}) is estimated by $1+d(k,0)\lesssim d(k,0)$, where we used $k\not=0$
in the last step.
Turning to (\ref{wg04}), we split the domain of summation into $\{d(k',0)\le d(k'',0)\}$
and $\{d(k',0)>d(k'',0)\}$. By (\ref{wg02}) (with $k'$ replaced by $k''$ for the second contribution),
both contributions are estimated by $1\lesssim d(k,0)$.
Finally addressing (\ref{wg05}) (which coincides with \eqref{newul}), we split the domain of summation into the same two sets.
On the first domain $\{d(k',0)\le d(k'',0)\}$, by the second inequality in (\ref{wg03})
we must have $d(k'',0)$ $\ge\frac{1}{2}d(k,0)$, so that by (\ref{wg02}), the corresponding
contribution is estimated as desired. On the second domain $\{d(k',0)>d(k'',0)\}$, we use that for the same reason
$d^4(k',0)d(k'',0)\ge \frac{1}{2}d(k,0)d^4(k'',0)$, so that by (\ref{wg02}) (with $k''$ playing the role of $k'$)
also this contribution is estimated as desired.

\medskip

\nd {\it Step 4}. We claim that for all $\ell\le 1$, $x\in\mathbb{R}^2$, and $T>0$ we have the estimate
\begin{align}
\langle(F^\ell_T(x))^2\rangle^\frac{1}{2}
&\lesssim(T^\frac{1}{3})^{-\frac{3}{4}},\label{f18}\\
\langle(\ell\frac{\partial}{\partial \ell} F^\ell_T(x))^2\rangle^\frac{1}{2}
&\lesssim\min\{(T^\frac{1}{3})^{-\frac{3}{4}},\ell(T^\frac{1}{3})^{-\frac{7}{4}}\}.\label{f19}
\end{align} 
We focus on \eqref{f19} and use the same type of arguments as in Step 2 of the proof of Lemma~\ref{L1}.
Since the distribution $\xi$ and its translation $\xi(\cdot+h)$ by some translation vector $h$ have
the same distribution under $\langle\cdot\rangle$, and since $v$ arises as the solution of a (linear)
constant-coefficient (pseudo-) differential operator with RHS $\xi$, also
$v$ and $v(\cdot+h)$ have the same distribution as fields. This shift-invariance carries over to $v_\ell$,
$\partial_2v_\ell$, $\partial_2 Rv_\ell$ and thus to $v_\ell \partial_2 Rv_\ell$,
$F^\ell$ and $F^\ell_T$. 
This implies that $\langle(\ell\frac{\partial}{\partial\ell}F_T^\ell(x))^2\rangle$ does not depend on $x$ so that
it is enough to show
\begin{align*}
\langle\int_{[0,1)^2}(\ell\frac{\partial}{\partial\ell}F_T^\ell)^2dx\rangle\lesssim\min\{(T^\frac{1}{3})^{-\frac{3}{2}},
\ell^2(T^\frac{1}{3})^{-\frac{7}{2}}\}.
\end{align*}
By Plancherel, this assumes the form of
\begin{align*}
\sum_{k\not=0}\langle|\ell\frac{\partial}{\partial\ell}F_T^\ell(k)|^2\rangle\lesssim\min\{(T^\frac{1}{3})^{-\frac{3}{2}},
\ell^2(T^\frac{1}{3})^{-\frac{7}{2}}\}.
\end{align*}
Because of the identity $\ell\frac{\partial}{\partial\ell}F^\ell_T(k)$ 
$=\psi_T(k)\ell\frac{\partial}{\partial\ell}F^\ell(k)$, the inequality $0\le \psi_T(k)\le\exp(-Td^3(k,0))$, 
and the previous steps in form of (\ref{clai}) this reduces to
\begin{align*}
\sum_{k\not=0}\exp(-2Td^3(k,0))&\min\{d^{-1}(k,0),\ell^2d(k,0)\}
\lesssim\min\{(T^\frac{1}{3})^{-\frac{3}{2}},
\ell^2(T^\frac{1}{3})^{-\frac{7}{2}}\}.
\end{align*}
The latter obviously splits into
\begin{align*}
\sum_{k\not=0}\exp(-2Td^3(k,0))d^{-1}(k,0)\lesssim(T^\frac{1}{3})^{-\frac{3}{2}},\nonumber\\
\sum_{k\not=0}\exp(-2Td^3(k,0))d(k,0)
\lesssim(T^\frac{1}{3})^{-\frac{7}{2}},
\end{align*}
which both follows from the effective dimension $\frac{5}{2}$ of $k$-space, cf.~(\ref{f16old}).

\medskip

\nd {\it Step 5}.
We claim that for all $\ell\le 1$, $1\leq p<\infty$, and $T>0$ we have the estimate
\begin{align}
\langle\|F^\ell_T\|^p\rangle^\frac{1}{p}
&\lesssim(T^\frac{1}{3})^{-\frac{3}{4}-\frac{1}{p}\frac{5}{2}},\label{f22}\\
\langle\|\ell\frac{\partial}{\partial \ell} F^\ell_T\|^p\rangle^\frac{1}{p}
&\lesssim\min\{(T^\frac{1}{3})^{-\frac{3}{4}-\frac{1}{p}\frac{5}{2}},
\ell(T^\frac{1}{3})^{-\frac{7}{4}-\frac{1}{p}\frac{5}{2}}\}.\label{f23}
\end{align}
The argument proceeds as in Step 3 of the proof of Lemma \ref{L1} with the one notable difference that we now can
no longer simply appeal to Gaussianity to get the inverse H\"older estimate
\begin{align*}
\langle|F_T^\ell(x)|^p\rangle^\frac{1}{p}\lesssim\langle(F_T^\ell(x))^2\rangle^\frac{1}{2},
\end{align*}
and the analogous statement for $\ell\frac{\partial}{\partial\ell}F_T^\ell$. However, such an estimate remains
true because $F_T^\ell(x)$ is a quadratic functional of $\xi$, and as such an element of what is called the Second Wiener Chaos.
In this situation, the estimate is known as Nelson's estimate, see \cite[Proposition 3.3]{MourratWeberXu} for a proof.

\medskip

\nd {\it Step 6}. We claim that for all $\eps\in (0,1)$, $1\le p<\infty$ and $\ell_0\le 1$ we have for every $T\leq 1$:
\begin{align}
\langle\sup_{\ell\le 1}\|F^\ell_T\|^p\rangle^\frac{1}{p}
&\lesssim (T^\frac{1}{3})^{-\frac{3}{4}-2\eps},\label{f32}\\
\langle\sup_{\ell,\ell'\le \ell_0}\|F^\ell_T-F^{\ell'}_T\|^p\rangle^\frac{1}{p}
&\lesssim\ell_0^\eps
(T^\frac{1}{3})^{-\frac{3}{4}-2\eps},\label{f31}
\end{align}
where now, the constant also depends on $\eps$.
This follows by the same argument as in Step 4 in the proof of Lemma \ref{L1}.

\medskip

\nd {\it Step 7}. By the same argument as in Step 5 in the proof of Lemma \ref{L1}, we have 
for any $0<2\eps<\frac{1}{4}$ and any $p\in [1, \infty)$:
\begin{align*}
\langle\sup_{\ell\le 1}[F^{\ell}]_{-\frac34-2\eps}^p\rangle^\frac{1}{p}&\lesssim 1,\quad 
\langle\sup_{\ell,\ell'\le\ell_0}[F^{\ell}-F^{\ell'}]_{-\frac34-2\eps}^p\rangle^\frac{1}{p}
\lesssim \ell_0^\eps,
\end{align*}
where the constant inside $\lesssim$ depends on $p$ and $\eps$.

\medskip

\nd {\it Step 8}. We now give the argument that the limit $F$ of $F^\ell$ in the Banach space defined through
the norm $\langle[\cdot]_{-\frac34-2\eps}^p\rangle^\frac{1}{p}$ is independent of the (normalized) symmetric Schwartz kernel 
$\phi$ that entered the definition of $F^\ell$ via convolution.
Take another (symmetric) Schwartz kernel $\tilde\phi$ and (with a benign misuse of language) 
denote by $v_{\tilde\ell}:=\tilde\phi_{\tilde\ell}*v$ the corresponding convolutions
with the rescaled kernel $\tilde\phi_{\tilde\ell}$ on $x_1$-scale $\tilde\ell$, cf. (\ref{f37}).
An inspection of the above proof,
in particular the claim \eqref{clai}, shows that the second estimate of this lemma holds with $v$ replaced by
$v_{\tilde\ell}$, since the relevant Fourier multipliers $G(k)$, cf. (\ref{f8}), and $\tilde G(k)$,
cf. (\ref{f1}), that lead from $\xi$ to $v$ and $\partial_2 Rv$, respectively, are dominated,
cf. (\ref{f36}), in an identical way. Hence also $F^{\ell\tilde\ell}$ 
$:=P(v_{\ell\tilde\ell}\partial_2 Rv_{\ell\tilde\ell})$ is a Cauchy sequence in $\ell$ and thus
has a limit, all w.r.t. to the Banach-space norm $\langle[\cdot]_{-\frac34-2\eps}^p\rangle^\frac{1}{p}$, 
and the limit is uniform w.r.t. to $\tilde\ell$ in view of the above mentioned uniform-in-$\tilde\ell$ 
estimates on the Fourier multipliers. On the other hand,
thanks to $\tilde\ell>0$, this limit is classical and given by $F^{\tilde\ell}$ 
$:=P(v_{\tilde\ell}\partial_2 Rv_{\tilde\ell})$.
By symmetry, the same double-indexed object $F^{\ell\tilde\ell}$ has a limit in $\tilde\ell$ (always w.r.t. 
$\langle[\cdot]_{-\frac34-2\eps}^p\rangle^\frac{1}{p}$), this limit is uniform in $\ell$, 
and for $\ell>0$ is given by $F^{\ell}$.
Thus by the triangle inequality w.r.t. $\langle[\cdot]_{-\frac34-2\eps}^p\rangle^\frac{1}{p}$, 
$F^{\tilde\ell}$ and $F^\ell$ get closer and closer
for $\ell$ and $\tilde\ell$ tending to zero. This implies that their respective limits
in $\langle[\cdot]_{-\frac34-2\eps}^p\rangle^\frac{1}{p}$ coincide, as desired. 

\end{proof}


\section{Appendix}

\subsection{The linearized energy}
\label{sec:lin_en}

For $\sigma>0$, we note that the linearized energy functional of \eqref{i10} on ${[0,1)^2}$:
\[ E_\text{lin}(u) = \int_{[0,1)^2} \Bigl( \lvert \partial_1 u \rvert^2 +  (\partial_2 u) \lvert \partial_1 \rvert^{-1} (\partial_2 u) 
- 2\sigma \xi \, u \Bigr) \, dx  \]
only admits critical points that have (negative) infinite energy (with positive probability). Indeed, if $u$ is a solution of the Euler-Lagrange equation $\mathcal{L}u =P\xi$, then we can explicitly solve it in Fourier space, obtaining
\[ u(k) =\sigma (k_1^2 + \lvert k_1 \rvert^{-1} k_2^2)^{-1} \xi(k) \, \textrm{ for }\,  
k_1\neq 0, \quad u(0, k_2)= 0  \, \textrm{ for }\,  k_2\in 2\pi\Z. \]
Since $\sum_{k_1\neq 0} \lvert k \rvert^{-2}$ diverges logarithmically,
\[ \langle E_\text{lin}(u) \rangle = - \sigma^2 \sum_{k\neq 0} \tfrac{\langle \lvert \xi(k) \rvert^2 \rangle}{\lvert k_1 \rvert^2 + \lvert k_1 \rvert^{-1} k_2^2}\sim -\sum_{k_1\neq 0} \tfrac{1}{\lvert k_1 \rvert^2 + \lvert k_1 \rvert^{-1} k_2^2} \leq -\sum_{k_1\neq 0} \tfrac{1}{\lvert k \rvert^2} = -\infty.  \]

\subsection{The anisotropic H\"older space $C^\alpha$ for $\alpha\in (0, \frac 32)$}
\label{sec:stand}

For the reader convenience, we give the following (standard) result for our anisotropic H\"older spaces $C^\alpha$, $\alpha>0$:

\begin{lem}\label{lem:linftyhoelder}
Let $u:[0,1)^2\to \R$ be a $1$-periodic function. If $\alpha\in (0, 1]$, then $\|u\|\lesssim [u]_\alpha$ provided that $u$ is of vanishing average.
\footnote{In particular, the assumption is satisfied if $u$ is of vanishing average in $x_1$.} If $\alpha\in (1, \frac32)$, then $\|\partial_1 u\|\lesssim [\partial_1 u]_{\alpha-1}\lesssim [u]_{\alpha}$ for every periodic function $u$. Moreover, if $\frac32>\alpha\geq \beta>0$, then $[u]_\beta\lesssim [u]_\alpha$, i.e., $C^\alpha\subset C^\beta$. Also, for two periodic functions $u\in C^\alpha$ and $f\in C^\beta$ on $[0,1)^2$ with $\alpha\geq \beta>0$, then $uf\in C^\beta$ with 
$$[uf]_\beta\lesssim \left([u]_\alpha+\left|\int_{[0,1)^2}u\, dx\right|\right) \left( [f]_\beta+\left|\int_{[0,1)^2}f\, dx\right|\right).$$ 
\end{lem}
\begin{proof} 

\nd {\it Step 1. If $\alpha\in (0, 1]$, then $\|u\|\lesssim [u]_\alpha$ for every periodic function $u$ of vanishing average.} Indeed, we have $u(x)-u(y) \leq [u]_\alpha d^\alpha(x,y) \lesssim [u]_\alpha$ for every $x,y\in {[0,1)^2}$. 
Integrating in $y\in {[0,1)^2}$, as $u$ has zero average, we obtain $u(x) \lesssim [u]_\alpha$ for all $x\in {[0,1)^2}$. 
Similarly, we have $-u(x) \lesssim [u]_\alpha$ on ${[0,1)^2}$, so that the first claim follows. 

\nd {\it Step 2. If $\alpha\in (1, \frac32)$, then $[\partial_1 u]_{\alpha-1}\lesssim [u]_{\alpha}$ for every periodic function $u$.} Indeed, let $u\in C^{\alpha}$ for $\alpha \in (1, \frac 32)$, i.e.,
$|u(y)-u(x)-\partial_1u(x)(y_1-x_1)|\leq [u]_\alpha d^\alpha(x,y)$ for every $x,y\in [0,1)^2$. Interchanging $x$ and $y$, by summation, we deduce that
$\big|[\partial_1u(x)-\partial_1u(y)](y_1-x_1)\big|\leq 2[u]_\alpha d^\alpha(x,y)$. In particular, for $x_2=y_2$, we deduce that 
$|\partial_1u(x_1, x_2)-\partial_1u(y_1, x_2)|\leq 2[u]_\alpha |x_1-y_1|^{\alpha-1}$ for every $x_1, y_1, x_2\in [0,1)^2$. We conclude that for every $x_1, x_2, y_2 \in [0,1)^2$: 
\begin{align*}
|\partial_1u(x_1, x_2)-\partial_1u(x_1, y_2)|&\leq |\partial_1u(x)-\partial_1u(y)|+|\partial_1u(x_1, y_2)-\partial_1u(y_1, y_2)|\\
&\leq 2[u]_\alpha \big(\frac{d^\alpha(x,y)}{|x_1-y_1|}+|x_1-y_1|^{\alpha-1}\big)\lesssim [u]_\alpha |x_2-y_2|^{2(\alpha-1)/3}
\end{align*} for $y_1$ chosen such that $|x_1-y_1|=|x_2-y_2|^{2\alpha/3}$. As $\partial_1 u$ is of vanishing average, Step 1 implies that $\|\partial_1 u\|\lesssim [\partial_1 u]_{\alpha-1}$.

\nd {\it Step 3. We prove that $[u]_\beta\lesssim [u]_\alpha$ for every periodic function $u$ if $0<\beta\leq \alpha <\frac32$.} Indeed, this is straightforward by Definition \ref{definitionCalpha} if $\alpha\leq 1$, respectively by Definition \ref{definitionCbeta} if $\beta> 1$. It remains to treat the case $\beta\leq 1<\alpha$. By Step 2, we already know that
$|u(x_1, x_2)-u(x_1, y_2)|\leq [u]_\alpha |x_2-y_2|^{2\alpha/3}$ and $|\partial_1u(x_1, x_2)-\partial_1u(y_1, x_2)|\leq 2[u]_\alpha |x_1-y_1|^{\alpha-1}$ for every $x_1, y_1, x_2, y_2\in [0,1)^2$. As $u$ is periodic and $\partial_1 u$ is of vanishing average, we deduce that $\|\partial_1 u\|\lesssim [u]_\alpha$. Therefore
\begin{align*}
|u(x)-u(y)|&\leq |u(x_1,x_2)-u(x_1, y_2)|+ |u(x_1,y_2)-u(y_1, y_2)|\\
&\leq [u]_\alpha |x_2-y_2|^{2\alpha/3}+\|\partial_1 u\| |x_1-y_1|\lesssim [u]_\alpha d^\beta(x,y).
\end{align*}

\nd {\it Step 4. We prove that $uf\in C^\beta$ for two periodic functions $u\in C^\alpha$ and $f\in C^\beta$ in $[0,1)^2$ with $\alpha\geq \beta>0$}. If $\beta\leq 1$, then by Step 3 we know that $u\in C^\alpha\subset C^\beta$ and we conclude that $uf\in C^\beta$ with the desired inequality for the product $uf$ because $\|u\|\lesssim [u]_\alpha+|\int_{[0,1)^2}u\, dx|$ and $\|f\|\lesssim [f]_\beta+|\int_{[0,1)^2}f\, dx|$. If $\beta>1$, then we have
for $x=(x_1,x_2)$, $y=(y_1, y_2)\in [0,1)^2$: 
\begin{align*}
&|(uf)(x)-(uf)(y)-\partial_1(uf)(y_1-x_1)|\\
&\leq |u(x)| [f]_\beta d^\beta(x,y)+|f(y)| [u]_\alpha d^\alpha(x,y)+|\partial_1 u(x)| |f(x)-f(y)| |y_1-x_1|\\
&\lesssim \|u\| [f]_\beta d^\beta(x,y)+\|f\| [u]_\alpha d^\beta(x,y)+\|\partial_1 u\| ([f]_\beta d^{\beta}(x,y)+\|\partial_1 f\| |y_1-x_1|) |y_1-x_1|\\
& \lesssim ([u]_\alpha+|\int_{[0,1)^2}u\, dx|)( [f]_\beta+|\int_{[0,1)^2}f\, dx|) d^\beta(x,y)
\end{align*}
as $\|\partial_1 u\|\lesssim [u]_\alpha$ and $\|\partial_1 f\|\lesssim [f]_\beta$.
\end{proof}

As a consequence we prove the following compactness result:

\begin{lem}\label{lem:compactness}
Let $(f_n)_n\subset \mathcal{D}'({[0,1)^2})$ be a sequence of periodic distributions satisfying the uniform bound $\limsup_{n\to\infty} [f_n]_\beta<\infty$ for some $\beta\in (-\frac32,\frac 32)\setminus\{0\}$. If $\beta\in (0, \frac32)$, we assume in addition,  that $\{\|f_n\|\}_n$ is uniformly bounded. Then there exists $f\in C^{\beta}$ in $[0,1)^2$ such that along a subsequence we have $f_n\wto f$ in $\mathcal{D}'({[0,1)^2})$ for $n\to\infty$ and $[f]_\beta \leq \liminf_{n\to\infty} [f_n]_\beta$.
\end{lem}

\begin{proof}[Proof of Lemma \ref{lem:compactness}] We can always assume that $[f_n]_{\beta}\to \liminf_{n} [f_n]_{\beta}=:C_0$ (by extracting eventually a subsequence).

If $\beta\in (0,1]$, then by Definition \ref{definitionCalpha} we know that $\{f_n\}$ is equicontinuous and uniformly bounded; thus the Ascoli theorem implies the conclusion. If $\beta\in (1,\frac32)$, by Lemma \ref{lem:linftyhoelder}, we know that $\{[f_n]_1\}$ is uniformly bounded as well as $\|\partial_1 f_n\|\lesssim [\partial_1 f_n]_{\beta-1}\lesssim C_0$; applying again the Ascoli theorem to $\{f_n\}$ and $\{\partial_1 f_n\}$, we deduce the conclusion.

We prove the case $\beta\in (-1,0)$ (the case $\beta\in (-3/2, 1]$ is similar). By Definition \ref{definitionCbeta}, there exist $g_n\in C^{1+\beta}({[0,1)^2},d)$ and $h_n\in C^{3/2+\beta}({[0,1)^2},d)$ (of vanishing average) such that 
$f_n=\int_{[0,1)^2}f_n\, dx+\partial_1 g_n+
\partial_2 h_n\in \mathcal{D}'({[0,1)^2})$ and
\[|\int_{[0,1)^2}f_n\, dx|+ [g_n]_{1+\beta} +  [h_n]_{3/2+\beta} \to C_0 \quad \textrm{as} \quad n\to 0  \]
and $\{\|g_n\|\}_n$ and $\{\|h_n\|\}_n$ are uniformly bounded.
By the previous cases, we know that $\int_{[0,1)^2}f_n\, dx\to A$, $g_n\to g$ and 
$h_n\to h$ for $n\to\infty$ along a subsequence where $g\in C^{1+\beta}$ and $h\in C^{3/2+\beta}$ in $[0,1)^2$ are of vanishing average and $[g]_{\beta+1} \leq \liminf_{n\to\infty} [g_n]_{\beta+1}$ as well as
$[h]_{\beta+3/2} \leq \liminf_{n\to\infty} [h_n]_{\beta+3/2}$. 
In particular, defining $f:=A+\partial_1 g + \partial_2 h$, Definition \ref{definitionCbeta} implies $f\in 
C^{\beta}$ in $[0,1)^2$ (of average $A$) with $[f]_\beta\leq C_0$ and $f_n \wto f$ distributionally because for every $\zeta\in C^\infty({[0,1)^2})$:
\begin{align*} 
\int_{[0,1)^2} (f_n-A) \zeta \, dx &= -\int_{[0,1)^2} (g_n \partial_1 \zeta + h_n \partial_2 \zeta) \, dx\\
& \stackrel{n\to\infty}{\to} -\int_{[0,1)^2} (g \partial_1 \zeta + h \partial_2 \zeta)  \, dx = \int_{[0,1)^2} (f-A) \zeta \,  dx.  
\end{align*} 
\end{proof}

\section*{Acknowledgements}
{The authors acknowledge contributions of Lukas D\"oring to an early version of this paper.}
R.I. acknowledges partial support by the ANR project ANR-14-CE25-0009-01. He thanks for the hospitality and support of the following institutions: Max-Planck-Institut f\"ur Mathematik in den Naturwissenschaften,  Institut des Hautes Etudes Scientiques (IHES) and Basque Center for Applied Mathematics, where part of this work was done. F.O. thanks for the hospitality and support of IHES, where part of this work was done.

\end{document}